\documentclass[14pt]{amsart}
\usepackage{cases}
\usepackage{bm}
\usepackage{amsfonts,amsmath,amssymb,amscd,bbm,amsthm,mathrsfs,dsfont}
\usepackage{mathrsfs}
\usepackage{pb-diagram}
\usepackage{color}
\usepackage{pifont}
\usepackage[all]{xy}
\usepackage{breqn}
\usepackage{epstopdf}
\usepackage{float}
\usepackage{xr}
\usepackage[colorlinks=true, linkcolor=blue, citecolor=red]{hyperref}
\newtheorem{Theorem}{Theorem}[section]
\newtheorem{Lemma}[Theorem]{Lemma}
\newtheorem{Definition}[Theorem]{Definition}
\newtheorem{Corollary}[Theorem]{Corollary}
\newtheorem{Proposition}[Theorem]{Proposition}
\newtheorem{Example}[Theorem]{Example}
\newtheorem{Remark}[Theorem]{Remark}

\newtheorem*{Theorem*}{Theorem}

\newenvironment{Proof}[1][Proof]{\begin{trivlist}
\item[\hskip \labelsep {\bfseries #1}]}{\flushright
$\Box$\end{trivlist}}

\title{Polytope realization of cluster structures}
\author{Jie Pan}

\address{Jie Pan
\newline Department of Mathematics, Faculty of Sciences, University of Sherbrooke, Sherbrooke, Quebec, Canada}
\email{jie.pan@usherbrooke.ca}

\date{version of \today}

\renewcommand{\P}{{\mathbb P}}
\newcommand{\Q}{{\mathbb Q}}
\newcommand{\N}{{\mathbb N}}
\newcommand{\R}{{\mathbb R}}
\newcommand{\F}{{\mathcal F}}
\newcommand{\A}{{\mathcal A}}

\newcommand{\T}{{\mathbb T}}
\newcommand{\Z}{{\mathbb Z}}
\newcommand{\p}{{^{\prime}}}

 \normalbaselines
\allowdisplaybreaks
\begin{document}

\begin{abstract}
  Based on the construction of polytope functions and several results about them in \cite{LP}, we take a deep look on their mutation behaviors to find a link between a face of a polytope and a sub-cluster algebra of the corresponding cluster algebra. This find provides a way to induce a mutation sequence in a sub-cluster algebra from that in the cluster algebra in totally sign-skew-symmetric case analogous to that achieved via cluster scattering diagram in skew-symmetrizable case by \cite{GHKK} and \cite{M}.

  With this, we are able to generalize compatibility degree in \cite{CL} and then obtain an equivalent condition of compatibility which does not rely on clusters and thus can be generalized for all polytope functions. Therefore, we could regard compatibility as an intrinsic property of variables, which explains the unistructurality of cluster algebras. According to such cluster structure of polytope functions, we construct a fan $\mathcal{C}$ containing all cones in the $g$-fan.

  On the other hand, we also find a realization of $G$-matrices and $C$-matrices in polytopes by the mutation behaviors of polytopes, which helps to generalize the dualities between $G$-matrices and $C$-matrices introduced in \cite{NZ} and leads to another polytope explanation of cluster structures. This allows us to construct another fan $\mathcal{N}$ which also contains all cones in the $g$-fan.
\end{abstract}
\maketitle
\bigskip
\tableofcontents

\section{Introduction}
The definition of a cluster algebra is formally introduced by Fomin and Zelevinsky in \cite{FZ1}. Since then there are many papers discussing its properties as well as its links with other topics. However, most researchers mainly focus on the skew-symmetrizable cases. In a long period of time, there is no efficient way to systemically study a totally sign-skew-symmetric cluster algebra until Ming Huang and Fang Li showed that an acyclic cluster algebra can always be covered by some cluster algebra associated to certain (infinite) locally finite simply laced acyclic quiver and use such covering to induce plenty of significant properties from (infinite) skew-symmetric case to acyclic sign-skew-symmetric case in \cite{HL}. However, for a cyclic totally sign-skew-symmetric cluster algebra, it is in general hard to know wether such covering exists.

In order to deal with a totally sign-skew-symmetric cluster algebra, inspired by many significant properties of a skew-symmetrizable cluster algebra, such as Laurent phenomenon (\cite{FZ1}), $F_{l;t}$ being a polynomial (\cite{FZ4}) and positivity (\cite{GHKK} for the most general case), as well as the works of Jiarui Fei, Lee, Li Li, Schiffler and Zelevinsky (\cite{LLZ},\cite{LLZ2},\cite{F},\cite{LLS}) which study the Laurent expression of a cluster variable with the help of its Newton polytope, the author construct a polytope function $\rho_h$ for each $h\in\Z^n$ which satisfies the above good properties locally (that is, under limited steps of mutations) with a similar philosophy to that in the construction of greedy elements in some sense, in a joint work with Fang Li (\cite{LP}). Then it turns out polytope functions not only satisfies these properties universally (that is, under any sequence of mutations), but also form a basis of the upper cluster algebra. Therefore in this way we obtain an efficient tool to talk about a totally sign-skew-symmetric cluster algebra.

To us, Newton polytopes provide a perspective to consider problems incident to cluster variables locally so that many results are easier or at least more direct to see since we only need to focus on the Laurent expression itself. Also, because we only need to focus on the Laurent expression of elements and we stay in the (upper) cluster algebra, this method requires less assumptions or restrictions, which often maintains feasible for a totally sign-skew-symmetric cluster algebra.

Essentially, the polytope method provides several ways to degenerate a problem to a lower rank case, by which therefore we usually only need to consider low rank cases. And it is well-known that a cluster algebra of rank 2 is always skew-symmetrizable. In other words, polytope method allows us to deal with a totally sign-skew-symmetric cluster algebra through skew-symmetrizable cluster algebras with simple enough cluster structures, which does not require any extra conditions. In \cite{LP}, such degeneration is mainly induced according to the construction of a polytope $N_h$ (it can be seen as a sum of sub-polytopes) or to the face complex structure of a polytope. In this paper, we mainly apply the second way. Hence we do not so rely on the concrete construction of a polytope $N_h$ here. And that is why the author decide not to repeat the construction of $N_h$ in this paper. Interested readers please refer to \cite{LP}.

During the study of polytope functions, we find they are closely related to theta functions introduced by Gross, Hacking, Keel and Kontsevich in \cite{GHKK}, so we wonder wether it is possible to recover a similar wall-chamber structure for a totally sign-skew-symmetric algebra where the cluster structure can be embedded, as Gross, Hacking, Keel and Kontsevich did in \cite{GHKK} for a skew-symmetrizable cluster algebra.

So in this paper, as a continuation of the work in \cite{LP} as well as a preparation of the above aim, we would like to further focus on the mutation of polytope functions to show several significant results which have already been proved for skew-symmetrizable cluster algebras also hold for totally sign-skew-symmetric cluster algebras.

In \cite{LP}, we mainly focus on the Laurent expression of every single polytope functions, which can be regarded as a generalization of cluster variables. While here in this paper, we try to cast lights on a generalization of clusters. Since we now get a collection of variables, we want to furthermore find some reasonable way to group them together. To us, the ``reasonable'' means it could in some sense explain the logic of clusters. This is my understanding of what cluster scattering diagrams do to cluster algebras. Therefore, we present two possible explanations by polytope functions. The first is done by finding an equivalent condition of compatibility, which naturally induces a generalization of cluster complex. While the second is given by a polytope realization of $G$-matrices and $C$-matrices, which can be regarded as two kinds of tropicalization of clusters.

More concretely, the paper is organized as follows.

In section 2, we recall some basic definitions about cluster algebras, (weighted) polytopes and polyhedral complexes. Then by taking the Nowton polytope of a Laurent polynomial, we can relate any homogenous Laurent expression in an upper cluster algebra with a weighted polytope.

In section 3, we recall some results in \cite{LP} about polytope functions and then based on the exchange relations of cluster variables, we summary how a face of a polytope $N_h$ is changed under a single mutation of the initial seed, which unveils how to induce a mutation sequence in a sub-cluster algebra from that in the original cluster algebra. This is the starting point of our discussion.
\begin{Theorem}
   For any face $S$ of $N_{h}$ and any $k\in[1,n]$, there is a face $S\p$ of $\mu_k(N_{h})$ correlating to $S$ under $\mu_k$, a cluster algebra $\A\p$ of rank $r$ and a vector $f\in\Z^r$, where $r=\max\{dim(S),dim(S\p)\}$ such that

  (i)\;$S$ is isomorphic to $N_{f}|_{\A\p}$ while $S\p$ is isomorphic to $\mu_k(N_{f}|_{\A\p})$ if either $S$ is not incident to any edge parallel to $e_{k}$ and $dim(S)=r-1$ or it contains a segment parallel to $e_{k}$.

  (ii)\;Both $S$ and $S\p$ are isomorphic to $N_f|_{\A\p}$ if $S$ does not contain a segment parallel to $e_{k}$.
\end{Theorem}

In the skew-symmetrizable case, Peigeng Cao and Fang Li proved in \cite{CL} that two clusters with some common cluster variables lie in the same connected component of the unlabeled exchange graph under freezing the common cluster variables based on a natural embedding of the cluster scattering diagram of a sub-cluster algebra into that of the original cluster algebra given by Muller in \cite{M}. As a weaker version of such embedding, the above theorem in section 3 is sufficient to generalize this result for totally sign-skew-symmetric cluster algebras in section 4.
\begin{Theorem}
  In a TSSS cluster algebra $\A$, if there is a subset $I\subseteq [1,n]$, a permutation $\sigma$ of $[1,n]$ and two clusters $X_{t}$ and $X_{t\p}$ in $\A$ satisfying $x_{i;t}=x_{\sigma(i);t\p}$ for any $i\in I$, then there exists a mutation sequence $\mu$ composed by $\{\mu_{i}|i\notin I\}$ such that $\mu(X_{t})=X_{t\p}$ up to a permutation of indices.
\end{Theorem}
Then we are able to generalize the compatibility degree introduced by Peigeng Cao and Fang Li in \cite{CL} which measures whether there is a cluster containing these two cluster variables.
\begin{Theorem}
  In a TSSS cluster algebra $\A$, for any $f\in\mathcal{U(A)}$ and two different clusters $X_t$ and $X_{t\p}$ both containing a cluster variable $x$, $d_{x}^t(f)=d_{x}^{t\p}(f)$. Hence it is rational to define $(f|x)=d_{x}^t(f)$ for any $f\in\mathcal{U(A)}$ and $x\in\mathcal{S}$.

  For any cluster variables $x$ and $x\p$, $(x\p|x)\geqslant-1$. Moreover, $(x\p|x)=-1$ if and only if $x\p=x$ while $(x\p|x)=0$ if and only if $x\p\neq x$ and there is a cluster containing both variables.
\end{Theorem}

In Section 5, we find such compatibility agrees with the equality $\rho_h\rho_{h\p}=\rho_{h+h\p}$ for polytope functions, and the latter condition does not rely on any cluster anymore.
\begin{Proposition}
  In a cluster algebra $\A$, let $g$ be a $g$-vector associated to some cluster variable and $h\in\Z^n$. Then $\rho_g\rho_{h}=\rho_{g+h}$ if and only if $(\rho_{h}|\rho_g)\leqslant0$.
\end{Proposition}
Therefore, we generalize the notion of compatibility and hence generalize a cluster as a maximal set of pairwise compatible self-compatible irreducible polytope functions, which induces a fan $\mathcal{C}$ as a generalization of the cluster complex. In particular, this explains the unistructurality of cluster algebras when restricted in the subset consisting of cluster variables.

In Section 6, we confirm a conjecture proposed by Siyang Liu in \cite{L} which suggests corresponding $C$-matrices associated to $B$ and $-B^\top$ respectively are sign-synchronic. Then with the help of this conjecture and the mutation of polytope functions, we generalize several amazing equations about $G$-matrices and $C$-matrices for totally sign-skew-symmetric cluster algebras given by Nakanishi and Zelevinsky under an assumption of column sign-coherence of $c$-vectors in \cite{NZ}, which present the duality of $G$-matrices and $C$-matrices and connect cluster algebras associated to $B$, $B^\top$ and $-B^\top$ respectively. Recently, Siyang Liu also showed in \cite{L} their philosophy can be combined with folding theory to deal with acyclic sign-skew-symmetric cluster algebras.
\begin{Theorem}
  For any $t,t\p\in\T_{n}$, the following equations hold

  (i)\;$(G^{B_{t\p};t\p}_{t})^{\top}=C^{B_{t}^{\top};t}_{t\p}$.

  (ii)\;$G^{B_t;t}_{t\p}G^{-B_{t\p};t\p}_t=I_n$.

  (iii)\;$G^{B_t;t}_{t\p}(C^{-B_{t}^{\top};t}_{t\p})^{\top}=I_n$ and $C^{B_t;t}_{t\p}C^{-B_{t\p};t\p}_t=I_n$.
\end{Theorem}

In Section 7, with the help of these results, we are able to realize $C$-matrices and $G$-matrices in the Newton polytopes of polytope functions as a collection of sets consisting of primitive vectors of certain edges or normals, which allows us to construct a complete fan $\mathcal{N}$ containing all cones in the $g$-fan for a given totally sign-skew-symmetric matrix analogue to the wall-chamber structure of the cluster scattering diagram for a skew-symmetrizable matrix.
\begin{Theorem}
  Let $\A$ be a TSSS cluster algebra and $\mu=\mu_{i_r}\cdots\mu_{i_1}$ be a mutation sequence from $\Sigma_{t_0}$ to $\Sigma_t$.

  (i)\;$C_{t_0}^{\epsilon B_t;t}=\overrightarrow{\mu(L^\epsilon)}$ for $\epsilon\in\{+,-\}$.

  (ii)\;$G_{t_0}^{-\epsilon B_t^\top;t}=\epsilon\overrightarrow{\mu(L^\epsilon)}^\bot$ for $\epsilon\in\{+,-\}$ when $n\geqslant2$.

  (iii)\;$\mu^\epsilon_{j_s}\neq\emptyset$ if and only if $i_s\notin \bigcup\limits_{i\in [1,n]\setminus I}supp(g_{i;t_0}^{-\epsilon B_{t(s-1)}^\top;t(s-1)})$ for any $s\in[1,r]$. In this case, there is a unique index $j\in I$ satisfying
  $$\pi_{[1,n]\setminus I}((G_{t_0}^{-\epsilon B_{t(s-1)}^\top;t(s-1)})^\top[i_s])=(G_{t_0}^{-\epsilon (B_{[1,n]\setminus I})_{t\p(s-1)}^\top;t\p(s-1)})^\top[j].$$

  (iv)\;There is a class of bijections $\{\varphi_s\}_{s\in[0,r]}$ between face complexes of polytopes
  \[\varphi_s:\quad\mu_{i_{s}}\cdots\mu_{i_1}(N)|_{\A(B)}\quad\overset{\cong}\longrightarrow\quad \mu_{i_{s}}\cdots\mu_{i_1}(N)|_{\A(-B^\top)}\]
  satisfying some properties so that the mutation behaviors of $S$ and $\varphi_s(S)$ are similar.
\end{Theorem}

\section{Preliminaries}

\subsection{Notations and notions}
In this paper, unless otherwise specified, we always fix the following notations and notions for convenience:

(i)\;For any $n\in\N$, $x\in\Z$,
\[[1,n]=\{1,2,\cdots,n\},\;\;\;\;\;\;\;
sign(x)=\left\{\begin{array}{cc}
                  0, & \text{if } x=0; \\
                  \frac{|x|}{x}, & \text{otherwise}.
                \end{array}\right.\;\;\;\;\text{and}\;\;\;\;\;
[x]_{+}=max\{x,0\}.\]
For a vector $\alpha=(\alpha_{1},\cdots,\alpha_{r})^\top\in\Z^{r}$, $[\alpha]_{+}=([\alpha_{1}]_{+},\cdots,[\alpha_{r}]_{+})^\top$.

(ii)\;We denote elements in $\R^n$ as column vectors and denote by $z_{1},\cdots,z_{n}$ the coordinates of $\R^{n}$. Let $\{e_{i}|i\in[1,n]\}$ be the standard basis of $\R^n$. In this paper, we will not add extra subscript $n$ to claim the basis is of $\R^n$ as we believe there is no risk of confusion.

(iii)\;Given a set of commutative variables $X=\{x_{1},\cdots,x_{n}\}$ and a vector $\alpha=(\alpha_{1},\cdots,\alpha_{n})^\top\in\R^{n}$, denote $X^{\alpha}:=\prod\limits_{i=1}^{n}x_{i}^{\alpha_{i}}$.

(iv)\;For a finite set $I$, denote by $|I|$ the number of elements in $I$.

(v)\;For any subset $I\subset[1,n]$, denote the canonical projection as
\[\pi_{I}:\quad\R^{n}\quad\longrightarrow\quad\R^{n-|I|}\quad\quad\quad\qquad\]
\[\alpha\quad\mapsto\quad (\alpha_{i})_{i\notin I}\;\;\;\;\;\]
and denote the canonical embedding as
\[\gamma_{I;v}:\quad\R^{n-|I|}\quad\longrightarrow\quad\R^{n}\quad\quad\quad\qquad\]
\[\alpha\quad\mapsto\quad \gamma_{I;v}(\alpha)\;\;\;\;\;\]
such that $\pi_I\circ\gamma_{I;v}(\alpha)=v$ while $\pi_{[1,n]\setminus I}\circ\gamma_{I;v}(\alpha)=\alpha$ for any $v\in\R^I$. When $I$ contains only one index $j$, we usually use $\gamma_{j;v}$ rather than $\gamma_{\{j\};v}$ for simplicity.

(vi) {\bf Points} imply lattice points in $\Z^{n}$ and {\bf Polytopes} imply those whose vertices are lattice points.

(vii)\;The {\bf partial order} ``$\leqslant$" in $\Z^{n}$ is defined as follows: for any $\alpha=(\alpha_1, \cdots, \alpha_n)^\top$ and $\beta=(\beta_1, \cdots, \beta_n)^\top$ in $\Z^{n}$, $\alpha\leqslant \beta$ if $\alpha_{i}\leqslant \beta_{i}$ for all $i\in[1,n]$.

\subsection{Cluster algebras}
In this subsection, we recall some preliminaries of cluster algebras mainly based on \cite{FZ4}.

An integer matrix $B=(b_{ij})_{n\times n}$ is called {\bf sign-skew-symmetric} if either $b_{ij}=b_{ji}=0$ or $b_{ij}b_{ji}<0$ for any $i,j\in[1,n]$. A \textbf{skew-symmetric} matrix is a sign-skew-symmetric matrix with $b_{ij}=-b_{ji}$ for any $i,j\in[1,n]$. A \textbf{skew-symmetrizable} matrix is a sign-skew-symmetric matrix such that there is a positive diagonal matrix $D$ satisfying that $DB$ is skew-symmetric.

For a sign-skew-symmetric matrix $B$ and any $k\in[1,n]$, we define another matrix $B'=(b_{ij}')_{n\times n}$ satisfying that for any $i,j\in[1,n]$,
  \begin{equation}\label{equation: mutation of B}
    b_{ij}'=\left\{\begin{array}{ll}
                             -b_{ij}, & \text{if } i=k\text{ or }j=k; \\
                             b_{ij}+sign(b_{ik})[b_{ik}b_{kj}]_{+}, & \text{otherwise}.
                           \end{array}\right.
  \end{equation}
Denote by $B'=\mu_k(B)$ the matrix obtained from $B$ by the \textbf{mutation} in direction $k$. If $B'$ is also sign-skew-symmetric, then we can mutate $B'$ in direction $k\p\in[1,n]$ to obtain $B''=\mu_{k\p}\mu_{k}(B)$. This goes on until we reach a non-sign-skew-symmetric matrix.

\begin{Definition}
  A sign-skew-symmetric matrix $B$ is called {\bf totally sign-skew-symmetric} if $\mu(B)$ is sign-skew-symmetric for any mutation sequence $\mu=\mu_{k_i}\cdots\mu_{k_1}$.
\end{Definition}

It is not hard to check that skew-symmetry (or skew-symmetrizability respectively) of a matrix is mutation invariant, so such matrices are totally sign-skew-symmetric. Therefore, we will call a cluster algebra skew-symmetric (or skew-symmetrizable respectively) if so are its exchange matrices.

It is conjectured in \cite{BFZ} that an acyclic sign-skew-symmetric matrix is totally sign-skew-symmetric. Later in \cite{HL}, Ming Huang and Fang Li proved this conjecture which provides a lot of totally sign-skew-symmetric matrices which are not skew-symmetrizable. Meanwhile, it is also confirmed in \cite{FZ1} the existence of mutation-acyclic totally sign-skew-symmetric matrices.

For convenience, we will denote a totally sign-skew-symmetric matrix (respectively, cluster algebra defined subsequently) briefly by a {\bf TSSS matrix} (respectively, {\bf TSSS cluster algebra}).

Let $(\P,\oplus,\cdot)$ be a semifield, that is, a free abelian multiplicative group endowed with a binary operation of (auxiliary) addition $\oplus$ which is commutative, associative and distributive with respect to the multiplication in $\P$, and $\F$ be the field of rational functions in $n\in\Z_{>0}$ independent variables with coefficients in $\Q\P$.
\begin{Definition}
  A \textbf{seed} in $\F$ is a triple $\Sigma=(X,Y,B)$ such that
  \begin{itemize}
    \item $X=(x_{1},x_{2},\cdots,x_{n})$ is an $n$-tuple whose components form a free generating set of $\F$;
    \item $Y=(y_{1},y_{2},\cdots,y_{n})$ is an $n$-tuple of elements in $\P$;
    \item $B$ is an $n\times n$ TSSS integer matrix.
  \end{itemize}
\end{Definition}
The above $X$ is called a \textbf{cluster} with \textbf{cluster variables} $x_{i}$, $y_{i}$ is called a \textbf{$Y$-variable} and $B$ is called an \textbf{exchange matrix}.

Sometimes when the order of cluster variables in a cluster matters, labeled seeds will be used instead of seeds by distinguishing two seeds up to a permutation of indices. In this paper, the order is not essential. But since we mainly focus on mutations here, in order to avoid confusions, we would like to fix orders of cluster variables. Unless otherwise specified, we mean labeled seeds when we say seeds.

\begin{Definition}
  For any seed $\Sigma=(X,Y,B)$ in $\F$ and $k\in[1,n]$, we say $\Sigma\p=(X\p,Y\p,B\p)$ is obtained from $\Sigma$ by the \textbf{mutation} in direction $k$ if $X\p=(x_1\p,\cdots,x_n\p)$, $Y\p=(y_1\p,\cdots,y_n\p)$ and $B\p$ satisfying
  \begin{equation}\label{equation: mutation of x}
            x_{j}\p=\left\{\begin{array}{ll}
                             \frac{y_{k}\prod\limits_{i=1}^{n}x_{i}^{[b_{ik}]_{+}}+\prod\limits_{i=1}^{n}x_{i}^{[-b_{ik}]_{+}}}{(y_{k}\oplus 1)x_{k}}, & \text{if } j=k;\\
                             x_{j}, & \text{otherwise}.
                           \end{array}\right.
  \end{equation}
  \begin{equation}\label{equation: mutation of y}
    y_{j}\p=\left\{\begin{array}{ll}
                             y_{k}^{-1} & \text{if }j=k; \\
                             y_{j}y_{k}^{[b_{kj}]_{+}}(y_{k}\oplus 1)^{-b_{kj}} & \text{otherwise}.
                           \end{array}\right.
  \end{equation}
  and $B\p=\mu_{k}(B)$. In this case, we write $\Sigma\p=\mu_{k}(\Sigma)$.
\end{Definition}
(\ref{equation: mutation of B}), (\ref{equation: mutation of x}) and (\ref{equation: mutation of y}) are called the mutation formula of exchange matrices, cluster variables and $Y$-variables respectively. It can be checked that $\Sigma\p$ is also a seed and the seed mutation $\mu_{k}$ ia an involution, that is, $\mu_{k}\mu_{k}(\Sigma)=\Sigma$.
\begin{Definition}
  Let $\T_{n}$ be the $n$-regular tree whose $n$ edges emanating from the same vertex are labeled bijectively by $[1,n]$. A \textbf{cluster pattern} is an assignment of a seed to each vertex of $\T_{n}$ such that if two vertices are connected by an edge labeled $k$, then the seeds assigned to them can be obtained from each other by the mutation in direction $k$.
\end{Definition}
In this paper, the seed assigned to a vertex $t\in\T_n$ is denoted by $\Sigma_{t}=(X_{t},Y_{t},B_{t})$ with
\[X_{t}=(x_{1;t},x_{2;t}\cdots,x_{n;t}),\quad Y_{t}=(y_{1;t},y_{2;t}\cdots,y_{n;t})\quad and\quad B_{t}=(b_{ij}^{t})_{i,j\in[1,n]}.\]

Now we are ready to introduce the definition of cluster algebras.
\begin{Definition}
  Given a cluster pattern, let $\mathcal{S}=\{x_{i;t}\in\F\mid i\in[1,n],t\in\T_{n}\}$ be the set consisting of all cluster variables. The \textbf{cluster algebra} $\mathcal{A}$ associated with the given cluster pattern is the $\Z\P$-subalgebra of $\F$ generated by $\mathcal{S}$.
\end{Definition}

Since up to a cluster isomorphism, $\A$ is uniquely determined by the initial exchange matrix $B$, we also denote it as $\A(B)$. In this paper, when we say a cluster algebra, we always mean a TSSS cluster algebra unless otherwise specified.

\begin{Definition}\cite{FZ1}
  An (labeled) exchange graph $E(\A)$ associated to a cluster algebra $\A$ is an $n$-regular graph whose vertices correspond to (labeled) seeds in $\A$ and whose edges correspond to mutations.
\end{Definition}
According to definition, the (labeled) exchange graph is a quotient graph of $\T_n$ via (labeled) seeds equivalence.

It can be seen from definition that a cluster algebra $\A$ depends on the choice of semifield $\P$. There is a special semifield which plays an extremely important role in cluster theory.
\begin{Definition}
  The tropical semifield $Trop(u_{1},u_{2},\cdots,u_{l})$ is the free abelian multiplicative group generated by $u_{1},u_{2},\cdots,u_{l}$ endowed with addition $\oplus$ defined as
        $$\prod\limits_{j=1}^{l}u_{j}^{a_{j}}\oplus\prod\limits_{j=1}^{l}u_{j}^{b_{j}}=\prod\limits_{j=1}^{l}u_{j}^{min(a_{j},b_{j})}.$$
\end{Definition}

In particular, we say a cluster algebra $\A$ is of \textbf{geometry type} if $\P$ is a tropical semifield. In this case, we usually also denote $\P$ as $Trop(x_{n+1},x_{n+2},\cdots,x_{m})$. Then according to the definition, $y_{j;t}$ is a Laurent monomial in $x_{n+1},x_{n+2},\cdots,x_{m}$ for any $j\in[1,n],t\in\T_{n}$. Hence we can define $b_{ij}^{t}$ for $i\in[n+1,m],j\in[1,n]$ such that
\[y_{j;t}=\prod\limits_{i=n+1}^{m}x_{i}^{b_{ij}^{t}}.\]
Denote $\tilde{B}_{t}=(b_{ij}^{t})_{i\in[1,m],j\in[1,n]}$ and $\tilde{X}_{t}=(x_{1;t},\cdots,x_{n;t},x_{n+1},\cdots,x_{m})$. Then the seed assigned to $t$ can also be represented as $(\tilde{X}_{t},\tilde{B}_{t})$. The mutation formula of $\tilde{B}$ in direction $k$ is the same as (\ref{equation: mutation of B}) (with $i\in[1,m]$ now) while that of $\tilde{X}$ becomes
\[{x_{j}}\p=\left\{\begin{array}{ll}
                             \frac{\prod\limits_{i=1}^{m}x_{i}^{[b_{ik}]_{+}}+\prod\limits_{i=1}^{m}x_{i}^{[-b_{ik}]_{+}}}{x_{k}}, & \text{if } j=k;\\
                             x_{j}, &  \text{otherwise}.
                           \end{array}\right.\]

\begin{Definition}
  A cluster algebra is said to have {\bf  principal coefficients} at vertex $t_{0}$ if $\P=Trop(Y_{t_{0}})$.
\end{Definition}
Hence a cluster algebra having principal coefficients at some vertex is of geometric type. If we use the form $(\tilde{X},\tilde{B})$ to represent a seed, an equivalent definition of having principal coefficients at vertex $t_{0}$ is that the matrix $\tilde{B}_{t_{0}}$ assigned to $t_{0}$ has the form $\tilde{B}_{t_{0}}=\left(\begin{array}{c}
                                                                                                                                                  B_{t_{0}} \\
                                                                                                                                                  I
                                                                                                                                               \end{array}\right),$
where $I$ is the $n\times n$ identity matrix.

\begin{Definition}
  In a cluster algebra with pricipal coefficients, for any $t\in\T_{n}$ and $i\in[1,n]$, denote by $C^{t_0}_{t}=(c_{ij}^t)$ the $n\times n$ matrix obtained from $\tilde{B}_{t}$ by deleting the first $n$ rows and by $c^{t_0}_{i;t}$ the $i$-th column \footnote{In \cite{LP}, we present vectors as row vectors for the convenience of coordinate typing. However, in this paper we would like to present vectors as column vectors in accordance with the notations in \cite{FZ4} as here we do not need to write concrete coordinates frequently and also because we will focus on relations between $G$-matrices and $C$-matrices in section 4, we would like to keep them the same as their original definitions to avoid confusion.} of $C^{t_0}_{t}$. They are called the \textbf{$C$-matrix} and the $i$-th \textbf{$c$-vector} associated to $X_t$ with respect to $X_{t_0}$ respectively.
\end{Definition}

The following theorem is called \textbf{Laurent phenomenon}, which is the most fundamental result in cluster theory.

\begin{Theorem}[Laurent phenomenon]\label{Laurent phenomenon}
  \cite{FZ1, FZ2}Let $\A$ be a cluster algebra with an arbitrary seed $(X,Y,B)$, then any cluster variable of $\A$ can be expressed as a Laurent polynomial in $X$ over $\Z\P$.
\end{Theorem}

Thus for a seed $(X_{t_{0}},Y_{t_{0}},B_{t_{0}})$ and any cluster variable $x_{l;t}$ in $\A$, we can express $x_{l;t}$ as a Laurent polynomial in cluster $X_{t_{0}}$ as
\[x_{l;t}=\frac{P_{l;t}^{t_{0}}}{\prod\limits_{i=1}^{n}x_{i;t_{0}}^{d_{i}^{t_{0}}(x_{l;t})}},\]
where $P_{l;t}^{t_{0}}$ is a polynomial in $X_{t_{0}}$ without non-constant monomial factor. The denominator vector $d_{l;t}^{t_{0}}=(d_{1}^{t_{0}}(x_{l;t}),d_{2}^{t_{0}}(x_{l;t}),\cdots,d_{n}^{t_{0}}(x_{l;t}))^\top$ is called the \textbf{$d$-vector} of $x_{l;t}$ with respect to cluster $X_{t_{0}}$.

Moreover, when $\A$ has principal coefficients at $t_{0}$, it is proved in \cite{FZ4} that \[P_{l;t}^{t_{0}}\in\Z[x_{1,t_{0}},\cdots,x_{n,t_{0}};y_{1,t_{0}},\cdots,y_{n,t_{0}}].\]
Hence $F_{l;t}^{t_{0}}=P_{l;t}^{t_{0}}|_{x_{i;t_{0}}\rightarrow1,\forall i\in[1,n]}$ is a polynomial in $Y_{t_{0}}$, which is called the \textbf{$F$-polynomial} of $x_{l;t}$ with respect to $X_{t_{0}}$.

When $\A$ has principal coefficients at $t_{0}$, under the $\Z^{n}$-grading given by
$$deg(x_{i;t_0})=e_{i},\qquad deg(y_{i;t_0})=-b_{i}^{t_{0}}$$
for any $i\in[1,n]$, where $b_{i}^{t_{0}}$ is the $i$-th column of $B_{t_{0}}$, the Laurent expression of $x_{l;t}$ in $X_{t_{0}}$ is homogeneous with degree $g_{l;t}^{t_{0}}$ which is called the \textbf{$g$-vector} of $x_{l;t}$ with respect to $X_{t_{0}}$. Denote by $G^{t_0}_{t}=(g^{t_0}_{1;t},\cdots,g^{t_0}_{n;t})$ the \textbf{$G$-matrix} associated to $X_t$.

In \cite{FZ4}, according to the above definition and mutation formula (\ref{equation: mutation of x}) of cluster variables, Fomin and Zelevinsky provided the following recurrences of $g$-vectors:
\[g_{i;t_0}^{t_0}=e_i,\quad\forall i\in[1,n],\]
while
\begin{equation}\label{equation: mutation of g-vectors}
  g_{i;t\p}^{t_0}=\left\{\begin{array}{ll}
                     g_{i;t}^{t_0}, & \text{if $i\neq k$;} \\
                     -g_{k;t}^{t_0}+\sum\limits_{j=1}^{n}[\epsilon b_{jk}^t]_+g_{j;t}^{t_0}-\sum\limits_{j=1}^n [\epsilon c_{jk}^t]_+b_j^{t_0}, & \text{if $i=k$.}
                   \end{array}\right.
\end{equation}
for any vertices $t,t\p\in\T_n$ connected by an edge labeled $k\in[1,n]$, where $\epsilon=\pm1$.

In the sequel, for a cluster algebra $\A$, we will always denote by $t_0$ the vertex to which the initial seed assigned unless otherwise specified. And when a vertex is not written explicitly, we always mean the initial vertex $t_0$. For example, we use $X$, $x_{i}$, $P_{l;t}$ to denote $X_{t_{0}}$, $x_{i;t_{0}}$, $P_{l;t}^{t_{0}}$ respectively.

Sometimes we will also need to focus on some cluster algebra whose initial matrix is not $B_{t_0}$. In such situation, following \cite{FZ4}, we use superscripts to claim the initial matrix as well as the initial vertex. For example, we denote by $G^{B;t}_{t\p}$ the $G$-matrix associated to $X_{t\p}$ in the cluster algebra with initial matrix $B$ and initial vertex $t$.

The following theorem shows the importance of principal coefficients case in the study of cluster algebras.
\begin{Theorem}[Separation Theorem]\label{separation theorem}
  \cite{FZ4}For any cluster algebra $\A$ and vertices $t,t\p \in\T_{n}$, the cluster variable $x_{l;t}$ can be expressed as
  \begin{equation*}
    x_{l;t}=\frac{F_{l;t}^{t\p}|_{\F}(\hat{y}_{1;t\p},\cdots,\hat{y}_{n;t\p})}{F_{l;t}^{t\p}|_{\P}(y_{1;t\p},\cdots,y_{n;t\p})}X_{t\p}^{g_{l;t}^{t\p}},
  \end{equation*}
  where
 \begin{equation*}
  \hat{y}_{j;t\p}=y_{j;t\p}\prod\limits_{i=1}^{n}x_{i;t\p}^{b^{t\p}_{ij}}.
  \end{equation*}
\end{Theorem}

We denote $\hat{Y}_{t}=\{\hat{y}_{1;t},\cdots,\hat{y}_{n;t}\}$ for any $t\in\T_{n}$.

\begin{Definition}
  For any seed $\Sigma_{t}$ assigned to $t\in\T_{n}$, we denote by $\mathcal{U}(\Sigma_{t})$ the $\Z\P$-subalgebra of $\F$ given by
  \begin{equation*}
    \mathcal{U}(\Sigma_{t})=\Z\P[X_{t}^{\pm 1}]\cap\Z\P[X_{t_{1}}^{\pm 1}]\cap\cdots\cap\Z\P[X_{t_{n}}^{\pm 1}],
  \end{equation*}
  where $t_{i}$ is the vertex connected to $t$ by an edge labeled $i$ in $\T_{n}$ for any $i\in[1,n]$. $\mathcal{U}(\Sigma_{t})$ is called the {\bf upper bound} associated with the seed $\Sigma_{t}$. And $\mathcal{U}(\A)=\bigcap\limits_{t\in\T_{n}}\mathcal{U}(\Sigma_{t})$ is called the \textbf{upper cluster algebra} associated to $\A$.
\end{Definition}

Laurent phenomenon claims the inclusion $\A\subseteq\mathcal{U(A)}$.

Sequent definitions and notions in this subsection follow those in \cite{LP}.

In \cite{LP}, we construct $\rho_{h}\in\N[[\hat{Y}]]X^{h}$ for each $h\in\Z^{n}$, that is, $\rho_{h}$ is of the form $\sum\limits_{p\in\N^{n}}a_{p}\hat{Y}^{p}X^{h}$ and $\rho_{h}|_{x_{i}\rightarrow1,\forall i\in[1,n]}$ is a power series of $Y$, where $a_p\in\N$. In order to emphasis that we regard $X$ as variables while $Y$ as coefficients, we will slightly abuse the notation to denote $\rho_{h}\in\N[Y][[X^{\pm1}]]$ and call it a {\bf formal Laurent polynomial} in $X$ with coefficients in $\N[Y]$.

For any $t,t\p\in\T_{n}$ connected by an edge labeled $k$ and any homogeneous Laurent polynomial
$$f\in\Z[Y_{t\p}][X_{t\p}^{\pm 1}]\cap\Z[Y_{t}][X_{t}^{\pm 1}]\subseteq\Z Trop(Y_{t\p})[X_{t\p}^{\pm 1}]$$
with grading $h$, we naturally have $f=F|_{\F}(\hat{Y}_{t\p})X_{t'}^{h}$, where $F$ is obtained from the Laurent expression of $f$ in $X_{t\p}$ by specilizing $x_{i;t\p}$ to 1 for any $i\in[1,n]$. We modify it into the Laurent polynomial
\begin{equation}\label{modify}
\frac{F|_{\F}(\hat{Y}_{t\p})}{y_{k;t\p}^{[h_{k}]_{+}}}X_{t\p}^{h},
\end{equation}
and denote by $L^{t}(f)$ the Laurent expression of this modified form (\ref{modify}) in $X_{t}$ with coefficients in $Y_{t}$ belonging to $\Z Trop(Y_{t})$, where $h_k$ is the $k$-th entry of $h$. Applying $L^{t}$ on $f$ changes the semifield from $Trop(Y_{t\p})$ to $Trop(Y_{t})$, so we compensate for this change with dividing $y_{k;t\p}^{[h_{k}]_{+}}$ as in (\ref{modify}).

Then we can define $L^{t;\gamma}(f)=L^{t}\circ L^{t^{(1)}}\circ\cdots\circ L^{t^{(s)}}(f)$ for any path $\gamma=t-t^{(1)}-\cdots-t^{(s)}-t\p$ in $\T_{n}$ if $L^{t}\circ L^{t^{(1)}}\circ\cdots\circ L^{t^{(j)}}(f)\in\Z[Y_{t^{(j)}}][X_{t^{(j)}}^{\pm 1}]$ for $j\in[0,s]$. It can be verified that $L^{t;\gamma}(f)$ only depends on the endpoints $t$ and $t\p$ (see \cite{LP} for more details), so we usually omit the path in the superscript. Naturally, $L^t$ also makes sense for formal Laurent polynomials.

\begin{Definition}
  (i)\;A \textbf{cluster monomial} in $\A$ is a monomial in a cluster $X_{t}$ for some $t\in\T_{n}$.

  (ii)\;For any $k\in[1,n]$ and $t\in\T_{n}$, the \textbf{exchange binomial} in direction $k$ at $t$ is a binomial
  $$M_{k;t}=\frac{1}{y_{k;t}\oplus1}(y_{k;t}\prod\limits_{i=1}^n x_{i;t}^{[b_{ik}^t]_+}+\prod\limits_{i=1}^n x_{i;t}^{[-b_{ik}^t]_+})\in\Z\P[x_{1;t},\cdots,x_{k-1;t},x_{k+1;t},\cdots,x_{n;t}].$$
\end{Definition}

\begin{Definition}
  (i)\;For a Laurent monomial $f=aY^{p}X^{q}$ in $X$, we call $a$ (respectively, $aY^{p}$) the {\bf constant coefficient} (respectively, {\bf coefficient}) of $f$ and say $f$ is {\bf constant coefficient free} (respectively, {\bf coefficient free}) if $a=1$ (respectively, $aY^{p}=1$).

  (ii)\;For a Laurent polynomial $P\in\Z\P[X^{\pm 1}]$ and a constant coefficient free Laurent monomial $p$ with $\alpha,\beta\in\Z^{n}$, we denote by $co_{p}(P)$ the constant coefficient of $p$ in $P$.

  (iii)\;For any Laurent polynomial $P$, $P\p$ is called a \textbf{summand} of $P$ if for any constant coefficient free Laurent monomial $p$, either $0\leqslant co_{p}(P\p)\leqslant co_{p}(P)$ or $co_{p}(P)\leqslant co_{p}(P\p)\leqslant 0$. $P\p$ is called a \textbf{monomial summand} of $P$ if it is moreover a Laurent monomial.

  (iv)\;For a Laurent polynomial $P\in\Z\P[X_t^{\pm 1}]$, we denote by $deg_{x_{i;t}}(P)$ the $x_{i;t}$-\textbf{degree} of $P$, that is, the $\Z$-grading induced by
  \[deg_{x_{i;t}}(x_{j;t})=\left\{\begin{array}{ll}
                                    0, & \text{if }j\neq i; \\
                                    1, & \text{if $j=i$.}
                                  \end{array}\right.\]
  And we say $P$ is {\bf $x_{i;t}$-homogeneous} if it is homogeneous under this grading.
\end{Definition}

Following \cite{LLZ}, a Laurent polynomial $P$ in $X$ is called \textbf{universally positive} if $P\in\N\P[X_{t}^{\pm 1}]$ for any $t\in\T_{n}$. Moreover, a universally positive Laurent  polynomial is said to be \textbf{universally indecomposable} if it cannot be expressed as a sum of two nonzero universally positive Laurent  polynomials.

The following result is a simple corollary of a well-know theorem characterizing cluster variables in a cluster algebra of rank 2 in \cite{FZ1}.
\begin{Proposition}\label{rank 2 finite}
  Let $\A$ and $\A\p$ be cluster algebras associated to a $2\times2$ matrix $B$ and $-B^\top$ respectively. Then $x_{i;t}=x_{j;t\p}$ in $\A$ if and only if $x_{i;t}=x_{j;t\p}$ in $\A\p$ for any $i,j\in[1,2]$ and any $t,t\p\in\T_2$.
\end{Proposition}

\subsection{Polytopes and polyhedral complexes}
Next we briefly introduce some definitions about polytopes and polyhedral complexes mainly based on \cite{Z} and \cite{LP}.

\begin{Definition}
  (i)\;A \textbf{polyhedron} is an intersection of finitely many closed halfspaces in $\R^{n}$ for $n\in\N$.

  (ii)\;The \textbf{convex hull} of a finite set $V=\{\alpha_{1},\cdots,\alpha_{r}\}\subseteq \R^{n}$ is
  \[\text{conv}(V)=\{\sum\limits_{i=1}^{r}a_{i}\alpha_{i}\; |\; a_{i}\geqslant 0,\sum\limits_{i=1}^{r}a_{i}=1\},\]
  the \textbf{affine hull} of $V$ is
  \[\text{aff}(V)=\{\sum\limits_{i=1}^{r}a_{i}\alpha_{i}\; |\; a_i\in\R, \sum\limits_{i=1}^{r}a_{i}=1\},\]
  while the \textbf{cone} of $V$ is
  \[\text{cone}(V)=\{\sum\limits_{i=1}^{r}a_{i}\alpha_{i}\; |\; a_i\geqslant0\}.\]

  (iii)\;An (unweighted) \textbf{polytope} is the convex hull of a finite set of points in  $\R^{n}$ for some $n\in \N$, or equivalently, a polytope is a bounded polyhedron in $\R^{n}$ for $n\in\N$. The \textbf{dimension of a polytope} is the dimension of its affine hull.

  (iv)\;Let $N\subseteq \R^{n}$ be a polyhedron. For some chosen $w\in\R^{n}$ and $c\in\R$, a linear inequality $w^{\top}p\leqslant c$ is called {\bf valid} for $N$ if it is satisfied for all points $p\in N$. A \textbf{face} of $N$ is a set of the form
  \[S = N\cap \{p\in\R^{n}\; |\; w^{\top}p = c\},\]
  where $w^{\top}p\leqslant c$ is a valid inequality for $N$. The {\bf dimension of a face} is the dimension of its affine hull.

  (v)\;The \textbf{vertices}, \textbf{edges} and \textbf{facets} of a polyhedron $N$ are its faces of dimension $0$, $1$, and $dim N - 1$ respectively. Denote by $V(N)$ and $E(N)$ the set consisting of vertices and edges of $N$ respectively.

  (vi)\;The {\bf Minkowski sum} $N\oplus N\p$ of two polyhedra $N$ and $N\p$ is the polyhedron consisting of all points $p+q$ for points $p\in N$ and $q\in N\p$.
\end{Definition}

In \cite{LP} we modify the original definition of polytopes by associating weight to each lattice point.
\begin{Definition}\label{weight}
  For a polytope $N$, the {\bf weight} of a point $p\in N$ is the integer placed on this point, denoted as $co_p(N)$, or simply $co_{p}$ when the polytope $N$ is known clearly. A polytope $N$ equipped with weights is called a {\bf weighted polytope} if $N=\text{conv}(supp(N))$, where $supp(N)=\{p\in N|co_{p}(N)\neq0\}$ is called the support of $N$.
\end{Definition}

For convenience we set $co_{p}(N)=0$ if $p\notin N$.

Minkowski sum of polytopes can be extended to polytopes with weights. For two weighted polytopes $N$ and $N\p$, we define the weights of $N\oplus N\p$ as
\begin{equation*}
co_{q}(N\oplus N\p)=\sum\limits_{p+p\p=q}co_{p}(N)co_{p\p}(N\p)
\end{equation*}
for any point $q\in N\oplus N\p$.

In the sequel, since all the polytopes we concerned about are weighted, we will omit the word ``weighted'' and simply call them polytopes.

For any polytope $N$ embedded in $\R^{n}$ and $w\in\R^{n}$, denote by $N[w]$ the polytope obtained from $N$ by a translation along $w$.

Given a polytope $N$ with a vertex $p$, a \textbf{lattice generating set} of $N$ based on $p$ is a minimal set of vectors $V=\{v_1,\cdots,v_r\}$ satisfying that for any $i\in[1,n]$, there are two points $q_i$ and $q_i\p$ with non-zero weights on an edge of $N$ such that $q_i\p-q_i=v_i$ and for any point $p\p\in N$ with non-zero weight, there is unique $a_i\in\N$ such that
\[p\p=p+\sum\limits_{i=1}^{r}a_iv_i.\]
Denote
$$ldim(N)=\min\{|V|\mid V\text{ is a lattice generating set of }N\}.$$
Apparently, $ldim(N)\geqslant dim(N)$.

\begin{Definition}
  Given a polytope $N$ with a point $v\in N$ and $i_{1},\cdots,i_{r}\in[1,n]$, define $\{i_{1},\cdots,i_{r}\}$-\textbf{section} of $N$ at $v$ to be the convex hull of lattice points in $N$ whose $j$-th coordinates equal to that of $v$ for any $j\in[1,n]\setminus\{i_{1},\cdots,i_{r}\}$.
\end{Definition}

For any two points $p$ and $q$, denote $l(\overline{pq})$ to be the length of the segment $\overline{pq}$ connecting $p$ and $q$.

\begin{Definition}
(i)\; A map $\eta:\R^r\rightarrow\R^s$ is called \textbf{affine} if $\eta(p)=Ap+w$ for any $p\in\R^r$, where $A\in Mat_{s\times r}(\R)$ and $w\in\R^s$.

(ii)\; A \textbf{projection} $\tau$ of two polytopes $N$ and $N\p$ is a restriction of an affine map $\eta:\R^{dim(N)}\rightarrow\R^{dim(N\p)}$ satisfying that $\tau(N)=N\p$ and the weights associated to $p$ in $N$ and to $\tau(p)$ in $N\p$ are the same for any (not necessary lattice) points $p\in N$, where the weights of non-lattice points are set to be zero. A projection is an {\bf isomorphism} if it is bijective.
\end{Definition}

Assume the dimensions of $N$ and $N\p$ are $r$ and $s$ respectively, and they are embedded in $\R^r$ and $\R^s$ respectively. For each $i\in[1,r]$, there are two different (not necessary lattice) points $p,p\p\in N$ such that $p-p\p=l(\overline{pp\p})e_{i}$. Then a linear map $\tilde\tau$ is induced by a projection $\tau: N\rightarrow N\p$ as
\begin{equation}\label{tildetau}
  \begin{array}{ccc}
    \tilde{\tau}:\quad\R^{r}\quad&\longrightarrow&\quad\R^{s} \\
    \qquad\qquad e_{i}\qquad & \mapsto & \qquad \frac{\tau(p)-\tau(p\p)}{l(\overline{pp\p})}
  \end{array}
\end{equation}
In the later discussion, $N\p$ is often a face of some polytope with higher dimension $n$, so we usually slightly abuse the notation to use $\tilde{\tau}$ as the linear map:
$$\tilde{\tau}: \; \R^{r}\quad\longrightarrow\quad\R^{s}\quad\hookrightarrow\quad\R^{n}.$$
A polytope projection is called \textbf{non-negative} if $\tilde{\tau}(e_i)$ is a non-negative vector for any $i$.

\begin{Definition}
  A \textbf{polyhedral complex} $K$ is a set of polyhedra in $\R^n$ such that

  (a)\; any face of a polyhedron in $K$ is in $K$.

  (b)\; the intersection of any two polyhedra in $K$ is a common face of them.
\end{Definition}
The \textbf{underlying set} $|K|$ of a polytope complex $K$ is the union of all polyhedra in $K$. A \textbf{fan} is a polyhedral complex consisting of polyhedral cones.

\begin{Definition}
  A map of two polyhedral complexes $f:\quad K\rightarrow L$ is induced from a continuous map $f:\quad |K|\rightarrow|L|$ of their underlying set such that it is affine when restricted to any face of $K$. A map of polyhedral complexes is called bijective if it is a bijection as a map of sets.
\end{Definition}

Given a polytope $N\in\R^n$, there are several natural ways to construct related polyhedral complexes. Here we introduce two of them which will be used later.
\begin{Definition}
  (i)\; The face complex of a polytope $N$ is the complex consisting of all faces of $N$.

  (ii)\; The normal fan of a polytope $N$ is
  \[\mathcal{N}(N):=\{C_S|S\text{ is a non-empty face of }N\},\]
  where
  \[C_S=\{w\in\R^n|S\subseteq\{p\in N|w^\top p=\max\{q\in N|w^\top q\}\}\}\]
  for any non-empty face $S$ of $N$.
\end{Definition}

For two fans $K$ and $K$ in $\R^n$, denote by
\[K\wedge K\p=\{C\cap C\p|C\in K,\qquad C\p\in K\p\}\]
their \textbf{common refinement}. In particular, if $K$ and $K\p$ are normal fans of two polytopes $N$ and $N\p$, then their common refinement equals the normal fan of $N\oplus N\p$.

\subsection{Newton polytopes}
Following Definition \ref{weight}, we can obtain a correspondence from polytopes to Laurent polynomials in the following way:

(i)\; To a Laurent monomial $a_{v}Y^{v}\in\Z[Y^{\pm}]$, we associate a vector $v$ with weight $a_v$. Hence a Laurent polynomial $f(Y)=\sum\limits_{v\in\Z^{n}}a_{v}Y^{v}$ with $a_{v}\neq 0\in\Z$ corresponds to a set consisting of vectors $v$, which is called the \textbf{support} of $f(Y)$, with weights $a_{v}$.

(ii)\;Denote by $N(f)$ the convex hull of the support of $f(Y)$ with integers $a_{v}$ placed at lattice points $v$. Then, we set up the following bijection:

$$\{\text{Laurent polynomials}\; f(Y) \}\longleftrightarrow \{\text{weighted polytopes}\; N(f)\}$$
\\
In particular, polynomials in $Y$ correspond to polytopes lying in the non-negative quadrant.

For a principal coefficients cluster algebra $\A$ of rank $n$, when a vector $h\in\Z^{n}$ is given, the above bijection naturally induces a bijection maps a homogeneous Laurent polynomial $f(\hat{Y})X^{h}$ of degree $h$ to the weighted polytope $N(f)$ which corresponds to $f(Y)=f(\hat{Y})X^{h}|_{x_{i}\rightarrow 1,\forall i}$.

Hence we call $N(f)$ {\bf the Newton polytope} corresponding to the Laurent polynomial $f(Y)$ or to the homogeneous Laurent polynomial $f(\hat{Y})X^{h}$ (with respect to $h$). In particular, denote by $N_{l;t}$ the Newton polytope of $F_{l;t}$ associated to the cluster variable $x_{l;t}$ for any $l\in[1,n]$ and $t\in\T_n$.

Due to the above correspondence, when $h$ is clear, we also use point $p$ to represent its corresponding Laurent monomial $\hat{Y}^pX^h$ in the sequel for convenience.

\section{Polytope functions under a single mutation}
We briefly introduce some results in \cite{LP} which will be used in the following discussion.

Let $\A$ be a cluster algebra of rank $n$ with principal coefficients and $(X,Y,B)$ be its initial seed. For each $h\in\Z^n$, we can construct a homogeneous formal Laurent polynomial $\rho_{h}\in\N[Y][[X^{\pm 1}]]$ of degree $h$ (see Construction 3.15 %??
in \cite{LP} for details) called the \textbf{polytope function} associated to $h$ with respect to $X$. Then denote
\[\widehat{\mathcal{P}}=\{L^{t_0}(\rho_{h}^t)|h\in\Z^{n}\}\qquad\text{and}\qquad\mathcal{P}=\widehat{\mathcal{P}}\cap\N Trop(Y)[X^{\pm 1}],\]
where $t$ is an arbitrary vertex in $\T_n$. For people who are not familiar with polytope functions, it is all right to take $h$ as $g$-vectors of cluster monomials and thus consider the Laurent expressions of cluster monomials and their Newton polytopes instead.

\begin{Theorem}\label{results in LP}
  (i)\; For any $h\in\Z^{n}$ and any $k\in[1,n]$, there is
  \[h^{t_{k}}=h-2h_{k}e_{k}+h_{k}[b_{k}]_{+}+[-h_{k}]_{+}b_{k}\]
  such that $L^{t_{k}}(\rho_{h})=\rho^{t_{k}}_{h^{t_{k}}}$, where $t_{k}\in\T_{n}$ is the vertex connected to $t_{0}$ by an edge labeled $k$ and $h_{k}$ is the $k$-th entry of $h$.

  (ii)\;Both $\widehat{\mathcal{P}}$ and $\mathcal{P}$ are independent of the choice of $t$. In particular, $x_{l;t}=\rho_{g_{l;t}}$ for any $l\in[1,n]$ and $t\in\T_{n}$, so $\mathcal{P}$ contains all coefficient free cluster monomials.

  (iii)\;Let $S$ be an $r$-dimensional face of $N_{h}$ for $h\in\Z^{n}$ such that $\rho_{h}\in \mathcal{P}$. Then there is a vector $h\p\in\Z^{ldim(S)}$, a cluster algebra $\A\p$ with principal coefficients of rank $ldim(S)$ and a non-negative polytope projection $\tau: N_{h\p}|_{\A\p}\rightarrow S$. In particular, $\tau$ is an isomorphism when $ldim(S)=r$, and $\rho_{h\p}$ is a cluster monomial in $\A\p$ when $\rho_{h}$ is a cluster monomial in $\A$.

  When $S$ is moreover an $I$-section of $N_h$ for some $I\subseteq[1,n]$, then $ldim(S)=r$, we may choose the initial exchange matrix of $\A\p$ to be obtained from that of $\A$ by deleting rows and columns parameterized by indices not in $I$ and $\tilde{\tau}(e_i)=e_i$ for any $i\in I$.

  (iv)\;When $\rho_{h}\in\mathcal{P}$, $[d_{k}(\rho_{h})]_{+}$ equals the maximal length of edges of $N_{h}$ parallel to $e_k$ for any $k\in [1,n]$. Moreover when $\rho_{h}$ is a cluster variable, $d_{k}(\rho_{h})\geqslant-1$ and the equality holds if and only if $\rho_{h}=x_{k}$.
\end{Theorem}

\begin{Remark}
  For each face $S$ of $N_h$, there is a unique non-negative lattice generating set $V$ of $ldim(S)$ elements based on the unique minimal point $p$ of $S$ and the non-negative polytope projection $\tau$ is induced by $V$ and $p$ such that $\tilde{\tau}$ bijectively sends the standard basis to $V$. Then $h\p$ and $\A\p$ are properly chosen to make sure $\tau(N_{h\p}|_{\A\p})=S$.
\end{Remark}

Due to Theorem \ref{results in LP} (i), we have the following definition of polytope mutations.

\begin{Definition}
  The polytope $N_{h^{t_{k}}}^{t_{k}}$ of $\rho_{h^{t_{k}}}^{t_{k}}$ is called the \textbf{mutation of the polytope $N_{h}$ of $\rho_h$} in direction $k$, and it is denoted by $\mu_{k}(N_{h})$.
\end{Definition}

By definition, we can obtain $L^{t_k}(\rho_{h^{t_k}})$ from $\rho_h$ through following process:

For an polytope function $\rho_{h}$ and any $k\in[1,n]$, $\rho_{h}$ has a decomposition with respect to $x_{k}$-degree
\[\rho_{h}=\sum\limits_{s<0}x_{k}^{s}M_{k}^{-s}P_s(k)+\sum\limits_{s\geqslant0}x_{k}^{s}P_s(k),\]
where $P_s(k)\in\N[Y][[(X\setminus\{x_{k}\})^{\pm1}]]$. Hence due to the mutation formula (\ref{equation: mutation of x}) of cluster variables, the Laurent expression of $\rho_{h}$ in $X_{t_{k}}$ over $\N[[Y]]$ equals
\begin{equation}\label{equation: expression of rho_h in X_tk}
  \rho_{h}=\sum\limits_{s<0}x_{k;t_{k}}^{-s}P_s(k)+\sum\limits_{s\geqslant0}x_{k;t_k}^{-s}M_{k}^{s}P_s(k),
\end{equation}
where $t_k$ is the vertex connected to $t_0$ by an edge labeled $k$. Then $L^{t_{k}}(\rho_{h})$ is obtained from (\ref{equation: expression of rho_h in X_tk}) by dividing $y_{k}^{[h_{k}]_{+}}$ and then substituting Y-variables $y_{i}$ by Laurent monomials in $Y_{t_{k}}$ according to (\ref{equation: mutation of y}) as
\[L^{t_k}(\rho_{h})=\sum\limits_{s<0}x_{k;t_{k}}^{-s}P\p_s(k)+\sum\limits_{s\geqslant0}x_{k;t_k}^{-s}M_{k;t_k}^{s}P\p_s(k)\]
with $P\p_s(k)\in\N[Y_{t_k}][[(X_{t_k}\setminus\{x_{k;t_k}\})^{\pm1}]]$ (note that the semifield is changed from $Trop(Y)$ to $Trop(Y_{t_k})$ as mentioned in the previous section).

\begin{Definition}
  In the above settings, we say that under the mutation in direction $k$, $x_{k}^{s}M_{k}^{[-s]_{+}}p$ and $x_{k;t_k}^{-s}M_{k;t_k}^{[s]_{+}}p\p$ \textbf{correlate} to each other for any non-zero monomial summands $p$ of $P_{s}(k)$ and $p\p$ of $P\p_s(k)$ such that $p\p$ can be obtained from $p$ by dividing $y_{k}^{[h_{k}]_{+}}$ and then substituting Y-variables $y_{i}$ by Laurent monomials in $Y_{t_{k}}$ according to (\ref{equation: mutation of y}), and we also say two non-zero monomial summands of $x_{k}^{s}M_{k}^{[-s]_{+}}p$ and $x_{k;t_k}^{-s}M_{k;t_k}^{[s]_{+}}p\p$ respectively correlate to each other under the mutation in direction $k$.

  In this sense, we also say two faces $S_1$ and $S_2$ of $N_h$ and $\mu_k(N_h)$ respectively correlate to each other under the mutation in direction $k$ if for any point $q$ in $S_i$ with non-zero weight, there is $q\p$ in $S_{3-i}$ with non-zero weight correlating to $q$ under the mutation in direction $k$ for $i=1,2$.
\end{Definition}

In the perspective of polytopes, $\mu_{k}(N_{h})$ is obtained from $N_{h}$ by the following three steps:

(a)\;For each $k$-section of $N_{h}$, assume it is the segment $\overline{pq}$ with $q\geqslant p$ and $$co_{p\p}=\sum\limits_{i=0}^{l(\overline{pp\p})}a_{i}\binom{[-deg_{x_{k}}(\hat{Y}^{p}X^{h})]_{+}}{l(\overline{pp\p})-i}$$
with $a_{i}\in\N$ (in fact $a_{i}$ here equals $m_{k}(p+ie_{k})|_{N_{h}}\in\N$ with $s_{k}=1$ defined in \cite{LP}) for any point $p\p$ lying in $\overline{pq}$, then replace $\overline{pq}$ by the segment $\overline{pq\p}$ with weight $co\p_{p\p}=\sum\limits_{i=0}^{l(\overline{pp\p})}a_{i}\binom{[deg_{x_{k}}(\hat{Y}^{p}X^{h})]_{+}}{l(\overline{pp\p})-i}$ for any point $p\p$ lying in $\overline{pq\p}$, where $q\p=q+deg_{x_{k}}(\hat{Y}^{p}X^{h})e_{k}$. After the replacement of all $k$-sections, we get a new polytope $N$ as their convex hull.

(b)\;Translate $N$ along $-[h_{k}]_{+}e_{k}$ to get $N[-[h_{k}]_{+}e_{k}]$.

(c)\;$\mu_{k}(N_{h})=\phi(N[-[h_{k}]_{+}e_{k}])$, where
\[\phi:\quad\R^n\quad  \longrightarrow  \quad\R^n\]
is an $\R$-linear transformation determined by
\begin{equation*}
  \phi(e_{j})=\left\{\begin{array}{ll}
                       -e_{k}, & \text{if }j=k; \\
                       e_{j}+[b_{kj}]_{+}e_{k}, & \text{otherwise}.
                     \end{array}\right.
\end{equation*}
following (\ref{equation: mutation of y}).

According to Theorem \ref{results in LP} (iii), for any $r$-dimensional face $S$ of $N_{h}$, there is a vector $h\p\in\Z^{ldim(S)}$, a cluster algebra $\A\p$ with principal coefficients of rank $ldim(S)$ and a non-negative polytope projection $\tau: N_{h\p}|_{\A\p}\rightarrow S$.

(i)\;If $S$ contains an edge parallel to $e_k$, then there is $l\in[1,ldim(S)]$ such that $\tilde{\tau}(e_l)=e_{k}$, and by comparing the above three steps of how $N_{h}$ and $N_{h\p}|_{\A\p}$ are changed under the mutation in direction $k$ and in direction $l$ respectively, we can see $\tilde{\tau}$ induces a polytope projection $\tau\p$ mapping $\mu_{l}(N_{h\p}|_{\A\p})$\footnote{In fact, sometimes $\tau\p$ requires certain conditions on $\A\p$ which are not demanded by $\tau$ as the construction of $\mu_{l}(N_{h\p})$ may rely on some entries of the exchange matrix which have no effect on $N_{h\p}$. So $\tau\p$ may not be able to found for some choice of $\A\p$. However, as we only care about the fact that $\tau\p$ exists for certain choice of $\A\p$, we prefer to skip the details of replacing our choice of $\A\p$ and always assume we have chosen a proper one at the beginning. Same for the following cases.} to the face $S\p$ of $\mu_k(N_h)$ correlating to $S$ under $\mu_k$.

(ii)\;If $S$ is incident to a segment parallel to $e_{k}$ but it does not contain any segment parallel to $e_{k}$, then each point in $S$ is a vertex of the $k$-section of $N_h$ at this point. Moreover, because $S$ is a face, all points in $S$ is either the smaller vertices (represented by $p$ in step (a)) or the larger vertices (represented by $q$ in step (a)) of the $k$-sections. In the former case, $S$ is also a face of $N$. While in the latter case, according to step (a), for every point $q\in S$ with weight $co_{q}(N_{h})$ we get $q\p$ lying in the same face $S_N$ of $N$ with weight $co_{q}(N_{h})$ because $deg_{x_{k}}(\hat{Y}^{q}X^{h})=h_{k}+\sum\limits_{i=1}^{n}q_{i}b_{ki}$ linearly depends on the coordinate of $q$. Therefore $S$ is isomorphic to $S_N$. Also note that translation and $\phi$ in step (b) and (c) respectively keep isomorphisms. Hence in both cases the isomorphism and $\tau$ leads to a polytope projection from $N_{h\p}|_{\A\p}$ to a face $S\p$ of $\mu_k(N_h)$ correlating to $S$ under $\mu_k$.

(iii)\;If $S$ is not incident to any segment parallel to $e_{k}$, same as (ii) we can find a polytope projection from $N_{h\p}|_{\A\p}$ to a face $S\p$ of $\mu_k(N_h)$ correlating to $S$ under $\mu_k$. Moreover, if $deg_{x_k}(\hat{Y}^pX^h)>0$ for some point $p$ in $S$, there is another face $S^{\prime\prime}$ of $\mu_k(N_h)$ correlating to $S$ under $\mu_k$ with $ldim(S^{\prime\prime})=ldim(S)+1$. As $S^{\prime\prime}$ satisfies (i), we can find a vector $h^{\prime\prime}\in\Z^{ldim(S^{\prime\prime})}$, a cluster algebra $\A^{\prime\prime}$ with principal coefficients of rank $ldim(S^{\prime\prime})$ and a polytope projection $\tau\p: \mu_{ldim(S^{\prime\prime})}(N_{h^{\prime\prime}}|_{\A^{\prime\prime}})\rightarrow S^{\prime\prime}$ induced by $\tau$ such that $\tilde{\tau}\p$ sends the standard basis to the lattice generating set based on the minimal point in $S\p$.

More precisely, $h^{\prime\prime}=\gamma_{deg_{x_k}(\hat{Y}^pX^h);ldim(S^{\prime\prime})}(h\p)$ with $p$ being the minimal point in $S$, the initial exchange matrix of $\A^{\prime\prime}$ equals that of $\A\p$ after deleting their $ldim(S^{\prime\prime})$-th column and row while the $ldim(S^{\prime\prime})$-th row of the initial exchange matrix of $\A^{\prime\prime}$ are chosen to make $deg_{x_k}(\hat{Y}^qX^{h^{\prime\prime}}|_{\A^{\prime\prime}})=deg_{x_k}(\hat{Y}^{\tau(q)}X^h|_{\A})$ for any point $q\in N_{h\p}|_{\A\p}$.

(iv)\;If $S$ does not contain any edge but a segment connecting lattices with non-zero weights parallel to $e_k$, then $\tilde{\tau}(e_i)\neq e_k$ for any $i\in[1,ldim(S)]$ and for any vertex $p$ of $S$, $deg_{x_k}(\hat{Y}^pX_h)\geqslant0$ if $\pi_k(p)$ is a vertex of $\pi_k(S)$. As $deg_{x_{k}}(\hat{Y}^{q}X^{h})=h_{k}+\sum\limits_{i=1}^{n}q_{i}b_{ki}$ linearly depends on the coordinate of $\pi_k(q)$, $deg_{x_k}(\hat{Y}^qX_h)\geqslant0$ for any $q\in S$. Similar to (iii), when $deg_{x_k}(\hat{Y}^qX_h)=0$ for any $q\in S$, the face of $\mu_k(N_h)$ correlating to $S$ under $\mu_k$ is isomorphic to $S$, hence there is a polytope projection from $N_{h\p}|_{\A\p}$ to it; when $deg_{x_k}(\hat{Y}^pX^h)>0$ for some point $p$ in $S$, we get the face $S\p$ of $\mu_k(N_h)$ correlating to $S$ under $\mu_k$ with $ldim(S\p)=ldim(S)+1$, a vector $h^{\prime\prime}\in\Z^{ldim(S\p)}$, a cluster algebra $\A^{\prime\prime}$ and a polytope projection $\tau\p: \mu_{ldim(S\p)}(N_{h^{\prime\prime}}|_{\A^{\prime\prime}})\rightarrow S\p$ induced by $\tau$ such that $\tilde{\tau}\p$ sends the standard basis to the lattice generating set based on the minimal point in $S\p$.

Here $h^{\prime\prime}=\gamma_{deg_{x_k}(\hat{Y}^pX^h);ldim(S^{\prime\prime})}(h\p)$ with $p$ being the minimal point in $S$, the initial exchange matrix of $\A^{\prime\prime}$ equals that of $\A\p$ right multiplying $(I_n\;\alpha)$ after deleting their $ldim(S^{\prime\prime})$-th row while the $ldim(S^{\prime\prime})$-th row of the initial exchange matrix of $\A^{\prime\prime}$ are chosen to make $deg_{x_k}(\hat{Y}^qX^{h^{\prime\prime}}|_{\A^{\prime\prime}})=deg_{x_k}(\hat{Y}^{\tau(q)}X^h|_{\A})$ for any point $q\in N_{h\p}|_{\A\p}$, where $\alpha\in\N^{ldim(S)}$ satisfies $e_k=(\tilde{\tau}(e_1),\cdots,\tilde{\tau}(e_{ldim(S)}))\alpha$.

Thus we have the following result showing how a face of $N_h$ is changed under a single mutation.
\begin{Theorem}\label{mutation of faces}
  For any face $S$ of $N_{h}$ and any $k\in[1,n]$, there is a face $S\p$ of $\mu_k(N_{h})$ correlating to $S$ under $\mu_k$, a cluster algebra $\A\p$ of rank $r$ and a vector $f\in\Z^r$, where $r=\max\{ldim(S),ldim(S\p)\}$ such that

  (i)\;$S$ is a non-negative polytope projection of $N_{f}|_{\A\p}$ under $\tau$ while $S\p$ is a non-negative polytope projection of $\mu_l(N_{f}|_{\A\p})$ under $\tau\p$ if either $S$ is not incident to any edge parallel to $e_{k}$ or $S$ contains a segment parallel to $e_k$, where $l$ is the index such that $\tilde{\tau}(e_l)=e_k$ or $\tilde{\tau}\p(e_l)=e_k$.

  (ii)\;$S$ is isomorphic to $S\p$ and they are non-negative polytope projection of $N_f|_{\A\p}$ if $S$ does not contain a segment parallel to $e_{k}$.
\end{Theorem}

Therefore, by Theorem \ref{mutation of faces}, for any sequence of mutations $\mu=\mu_{i_{r}}\cdots\mu_{i_{1}}$ and any sequence of faces $S_0,\cdots,S_r$ satisfying $S_0$ is a face of $N_h$ and $S_{k}$ is a face of $\mu_{i_k}\cdots\mu_{i_1}(N_{h})$ correlating to $S_{k-1}$ under $\mu_{i_k}$ for any $k\in[1,r]$, there is a cluster algebra $\A\p$ of rank $s$ with principal coefficients at $t_0$, a vector $f\in\Z^s$ and a mutation sequence $\tilde{\mu}=\mu_{j_{r}}\cdots\mu_{j_{1}}$ such that there is a non-negative polytope projection $\tau_k:\mu_{j_k}\cdots\mu_{j_1}(N_{f}|_{\A\p})\rightarrow S_k$ for any $k\in[0,r]$, where $s=\max\{ldim(S_k)|k\in[0,r]\}$ and
\begin{equation*}
  \mu_{j_{k}}=\left\{\begin{array}{ll}
                       \mu_{l}, & \text{if either }\tilde{\tau}_{k-1}(e_l)=e_{i_{k}}\text{ and $S_{k-1}$ contains a segment parallel to }e_{i_k}\\ 
                       & \text{or }\tilde{\tau}_{k}(e_l)=e_{i_{k}}\text{ and $S_{k}$ contains a segment parallel to }e_{i_k};\\
                       \emptyset, & \text{if neither $S_{k-1}$ nor $S_k$ contains a segment parallel to $e_{i_k}$}.
                     \end{array}\right.
\end{equation*}
In general $\tilde{\mu}$ is not uniquely determined by $\mu$ and $S_0$ as there may be more than one faces of $\mu_{i_k}\cdots\mu_{i_1}(N_h)$ which are not isomorphic to each other correlating to $S_{k-1}$ under $\mu_{i_k}$.

\section{Compatibility of cluster variables}
As an application of Theorem \ref{mutation of faces}, we will show the definition of compatibility degree can be generalized to TSSS cluster algebras.

The notion of compatibility degree is introduced by Fomin and Zelevinsky in \cite{FZ2} for skew-symmetric cluster algebras of finite type to testify wether two cluster variables can be contained in a cluster, via a bijection between the set of cluster variables and almost positive roots of the corresponding root system. Later, Ceballos and Pilaud generalize the definition of compatibility degree using the $d$-vector of cluster variables for cluster algebras of finite type. In the latter definition, the key point is to show the entries of a $d$-vector does not depend on the choice of whole cluster (i. e., Theorem \ref{compatibility degree}). In \cite{CL}, Peigong Cao and Fang Li generalize the latter definition for skew-symmetrizable cluster algebras.

\begin{Lemma}\cite{CL}\label{two cluster monomials equal}
  In a TSSS cluster algebra, two cluster monomials $X^{\alpha}_{t}$ and $X^{\beta}_{t\p}$ equal if and only if there is a permutation $\sigma$ of $[1,n]$ such that $\alpha_i=\beta_{\sigma(i)}$ and $x_{i;t}=x_{\sigma(i);t\p}$ for any $i\in\{j\in[1,n]|\alpha_j\neq0\}$.
\end{Lemma}

\begin{Theorem}\label{a straight way}
  In a TSSS cluster algebra $\A$, if there is a subset $I\subseteq [1,n]$, a permutation $\sigma$ of $[1,n]$ and two clusters $X_{t}$ and $X_{t\p}$ in $\A$ satisfying $x_{i;t}=x_{\sigma(i);t\p}$ for any $i\in I$, then there exists a mutation sequence $\mu$ composed by $\{\mu_{i}|i\notin I\}$ such that $\mu(X_{t})=X_{t\p}$ up to a permutation of indices.
\end{Theorem}
\begin{proof}
  Let $\mu=\mu_{i_r}\cdots\mu_{i_1}$ be a mutation sequence such that $\mu(X_{t})=X_{t\p}$. It is trivial if $i_s\notin I$ for any $s\in[1,r]$. So assume $i_s\in I$ for some $s\in[1,r]$. Let $h\in\Z^n$ satisfy that $L^{t\p}(\rho^t_h)=\rho^{t\p}_{h^{t\p}}$ with $h^{t\p}\in\Z_{>0}^n$ and $N^{t}_{h}$ lies in the hyperplanes $z_i=0$ for any $i\in I$. This can be done by choosing $h^{t\p}=(h_1^{t\p},\cdots,h_n^{t\p})^\top\in\Z_{>0}^n$ with $h_{\sigma(i)}^{t\p}$ large enough for any $i\in I$.

  Then $N^{t\p}_{h^{t\p}}=\mu(N^t_h)$ equals the origin. By Theorem \ref{mutation of faces}, for a sequence of faces $N^t_h=S_0,\cdots,S_r$ such that $S_j$ is a face of $\mu_{i_j}\cdots\mu_{i_1}(N^t_h)$ correlating to $S_{j-1}$ under $\mu_{i_{j+1}}$ for $j\in[1,r]$, $\mu$ induces a sequence of mutations $\tilde{\mu}$ in some cluster algebra $\A\p$ of rank $s=\max\{ldim(S_j)|j\in[0,r]\}\geqslant n-|I|$ with principal coefficients and a vector $h\p\in\Z^s$ such that $N_{h\p}|_{\A\p}=N^t_h$, $\tilde{\mu}(N_{h\p})|_{\A\p}$ equals the origin and the matrix obtained from the initial exchange matrix of $\A\p$ by deleting rows and columns parameterized by indices in $I$ equals that obtained from $B_{t}$ by deleting rows and columns parameterized by indices in $I$.

  Denote $s=\min\{j\in[1,r]|i_j\in I\}$. Following from Theorem \ref{results in LP} (iii) and Theorem \ref{mutation of faces}, $S_j=\mu_{i_j}\cdots\mu_{i_1}(N^t_h)$ when $j\in[1,s-1]$. If we in particular choose $S_s$ to be the face correlating to $S_{s-1}$ with $dim(S_s)=dim(S_{s-1})$, according to Theorem \ref{results in LP} (iii) and Theorem \ref{mutation of faces}, $\tilde{\mu}_{i_s}=\emptyset$, which means in this case the length of $\tilde{\mu}$ is strictly less than that of $\mu$.

  Therefore, for a mutation sequence $\mu$ in $\A$ such that $\mu(N^t_h)|_{\A}$ equals the origin, once there is $j\in[1,r]$ satisfying $i_j\in I$, we could find a mutation sequence $\tilde{\mu}$ of strictly less length in $\A\p$ and a vector $f$ such that $N^t_f|_{\A\p}=N^t_h|_{\A}$ lies in the hyperplanes $z_i=0$ for any $i\in I$ when embedded in $\R^n$, $\tilde{\mu}(N^t_f)|_{\A\p}$ equals the origin and the matrix obtained from the initial exchange matrix of $\A\p$ by deleting rows and columns parameterized by indices in $I$ equals that obtained from $B_{t}$ by deleting rows and columns parameterized by indices in $I$.

  Since the length of $\mu$ is finite, so iteratively we could finally get a mutation sequence $\tilde{\mu}$ in $\A\p$ composed by $\{\mu_{i}|i\notin I\}$ and a vector $f$ such that $N^t_f|_{\A\p}=N^t_h|_{\A}$ lies in the hyperplanes $z_i=0$ for any $i\in I$ when embedded in $\R^n$, $\tilde{\mu}(N^{t}_{f})|_{\A\p}$ equals the origin and the initial exchange matrix of $\A\p$ equals the matrix obtained from $B_{t}$ by deleting rows and columns parameterized by indices in $I$. Hence $\tilde{\mu}(N^t_h)|_{\A}$ also equals the origin, which means $\rho^t_h$ is a cluster monomial both in $X_{t\p}$ and in $\tilde{\mu}(X_t)$. Then following from the assumption $h^{t\p}\in\Z_{>0}^n$ and Lemma \ref{two cluster monomials equal}, $\tilde{\mu}(X_{t})=X_{t\p}$ up to a permutation of indices.
\end{proof}

For any $f\in\mathcal{U(A)}$ and any cluster $X_t$, $f$ can be expressed as a Laurent polynomial in $X_t$, say $f=\frac{P^t(f)}{X_t^{d^t(f)}}$, where $P^t(f)\in\Z\P[X_t]$ can not be divided by any cluster variable in $X_t$ and $d^t(f)=(d_{x_{1;t}}^t(f),\cdots,d_{x_{n;t}}^t(f))\in\Z^n$. It is proved amazingly that as a corollary of Theorem \ref{a straight way}, $d_{x_{i;t}}^t(f)$ only depends on $f$ and $x_{i;t}$ but $X_t$ for skew-symmetrizable cluster algebras in \cite{CL}. The argument remains valid for TSSS cluster algebras as Theorem \ref{a straight way} has been generalized to TSSS cluster algebras.

\begin{Theorem}\label{compatibility degree}
  In a TSSS cluster algebra $\A$, for any $f\in\mathcal{U(A)}$ and two different clusters $X_t$ and $X_{t\p}$ both containing a cluster variable $x$, $d_{x}^t(f)=d_{x}^{t\p}(f)$.
\end{Theorem}
\begin{proof}
  Assume $x=x_{l;t}$. According to Theorem \ref{a straight way}, there is a mutation sequence $\mu$ consisting of mutations at directions other than $l$ satisfying that $\mu(X_t)$ equals $X_{t\p}$ up to a permutation of indices. So without loss of generality we assume $X_{t\p}=\mu(X_t)$ and $x=x_{l;t}=x_{l;t\p}$.

  It is enough to show the result holds when $\mu$ is a single mutation, then an induction on the length of $\mu$ completes the proof.

  Assume $t$ and $t\p$ are connected by an edge labeled $k\neq l$ and $f=\frac{P^t(f)}{X_t^{d^t(f)}}=\frac{P^{t\p}(f)}{X_{t\p}^{d^{t\p}(f)}}$. According to (\ref{equation: mutation of x}), $x_{i;t\p}=x_{i;t}$ when $i\neq k$ and $x_{k;t\p}=\frac{M_{k;t}}{x_{k;t}}$, where $M_{k;t}=X_{t}^{[b_{k}^t]_+}+X_{t}^{[-b_{k}^t]_+}\in\N\P[X_t\setminus\{x_{k;t}\}]$. As $f\in\mathcal{U(A)}$, we can denote $P^{t\p}(f)=\sum\limits_{s=0}^{deg_{x_{k;t\p}}(P^{t\p}(f))}x_{k;t\p}^sM_{k;t}^{[d_{x_{k;t\p}}^{t\p}(f)-s]_+}P_s^{t\p}(f)$, where $P_s^{t\p}(f)\in\Z\P[X_{t\p}\setminus\{x_{k;t\p}\}]$. Then as a Laurent polynomial in $\Z\P[X_t^{\pm1}]$,
  \[\frac{P^t(f)}{X_t^{d^t(f)}}=\frac{P^{t\p}(f)}{X_{t\p}^{d^{t\p}(f)}}|_{x_{k;t\p}\rightarrow\frac{M_{k;t}}{x_{k;t}}}= \frac{\sum\limits_{s=0}^{deg_{x_{k;t\p}}(P^{t\p}(f))}x_{k;t\p}^{deg_{x_{k;t\p}}(P^{t\p}(f))-s}M_{k;t}^{[s-d_{x_{k;t\p}}^{t\p}(f)]_+}P_s^{t\p}(f)} {x_{k;t}^{deg_{x_{k;t\p}}(P^{t\p}(f))-d_{x_{k;t\p}}^{t\p}(f)}\prod\limits_{i\neq k}x_{i;t}^{d_{x_{i;t\p}}^{t\p}(f)}}.\]
  In particular, $d_{x}^t(f)=d_{x}^{t\p}(f)$ as $l\neq k$.
\end{proof}

Therefore, as constructed in \cite{CL}, we have a map called \textbf{compatibility degree}
\[(-|-):\quad \mathcal{U(A)}\times\mathcal{S}\quad \longrightarrow \quad\Z\]
such that $(f|x)=d_{x}^t(f)$ for any $f\in\mathcal{U(A)}$ and $x\in\mathcal{S}$, where $t$ is an arbitrary vertex satisfying $X_t$ contains $x$. The next proposition explains why it is so named.
\begin{Proposition}\label{compatibility degree equals 0}
  Let $\A$ be a TSSS cluster algebra with two cluster variables $x$ and $x\p$, then $(x\p|x)\geqslant-1$. Moreover, $(x\p|x)=-1$ if and only if $x\p=x$ while $(x\p|x)=0$ if and only if $x\p\neq x$ and there is a cluster containing both variables.
\end{Proposition}
\begin{proof}
  By definition $(x\p|x)$ is the corresponding entry of a $d$-vector, hence according to Theorem \ref{results in LP} (iv) $(x\p|x)\geqslant-1$ and the equality holds if and only if $x\p=x$.

  If $x\p\neq x$ and there is a cluster containing both variables, then obviously $(x\p|x)=0$. Next we assume $(x\p|x)=0$. Choose a cluster $X_t$ containing $x$ and a cluster $X_{t\p}$ containing $x\p$. Denote by $g^t_{x\p}$ the $g$-vector of $x\p$ with respect to $X_t$ and by $\mu$ a mutation sequence such that $\Sigma_{t\p}=\mu(\Sigma_t)$. Hence $\mu(N^t_{g^t_{x\p}})$ is the origin.

  Since $(x\p|x)=0$, by Theorem \ref{results in LP} (iv), $N^t_{g^t_{x\p}}$ lies in the hyperplane $z_l=0$ for $l$ satisfying $x_{l;t}=x$. Therefore, analogous to the proof of Theorem \ref{a straight way}, we can find a mutation sequence $\tilde{\mu}$ consisting of mutations in $\{\mu_j|j\neq l\}$, a cluster algebra $\A\p$ of rank $n-1$ with principal coefficients and a vector $f=\pi_l(g^t_{x\p})$ such that $N^t_{f}|_{\A\p}=N^t_{g^t_{x\p}}$, $\tilde{\mu}(N^t_f)|_{\A\p}$ is the origin and the initial exchange matrix of $\A\p$ equals the matrix obtained from $B_t$ by deleting the $l$-th column and row. Hence we also have $\tilde{\mu}(N_{g^t_{x\p}})|_{\A}$ is the origin, which means $\tilde{\mu}(\Sigma_t)$ contains $x\p$ due to Lemma \ref{two cluster monomials equal}. On the other hand, since $\mu_l$ never appears in $\tilde{\mu}$, $\tilde{\mu}(\Sigma_t)$ contains $x$ as well.
\end{proof}

\section{Compatibility of polytope functions}
Although the compatibility degree defined via $d$-vectors does not depend on the choice of cluster as shown in the last section, it does rely on clusters. In this section, based on the discussion of last section and several results in \cite{LP} for polytope functions, we will define the compatibility of polytope functions and then show that it is equivalent to the original compatibility when restricted to cluster variables. Because the former definition does not rely on clusters at all, it provides a new intrinsic perspective to look at compatibility of cluster variables and unistructurality of cluster algebras, which is conjectured by Assem, Schiffler and Shramchenko in \cite{ASS} and proved for skew-symmetrizable cluster algebras by Peigeng Cao and Fang Li in \cite{CL2}.

\begin{Definition}
  We say two polytope functions $\rho_h$ and $\rho_{h\p}$ are \textbf{compatible} with each other if $\rho_h\rho_{h\p}=\rho_{h+h\p}$.
\end{Definition}

The following proposition shows the compatibility of polytope functions coincides with that of cluster variables when restricted to cluster variables.
\begin{Proposition}\label{compatibility degree and multiplication}
  In a cluster algebra $\A$, let $g$ be a $g$-vector associated to some cluster variable and $h\in\Z^n$. Then $\rho_g\rho_{h}=\rho_{g+h}$ if and only if $(\rho_{h}|\rho_g)\leqslant0$.
\end{Proposition}
\begin{proof}
  According to Theorem \ref{results in LP} (i) and the definition of compatibility degree, this proposition does not depend on the choice of the initial seed. So we may assume $\rho_g$ is contained in the initial cluster as $x_i$.

  If $(\rho_{h}|\rho_g)>0$, then $[d_{\rho_g}(\rho_g\rho_{h})]_+=[d_{\rho_g}(\rho_{h})-1]_+=d_{\rho_g}(\rho_{h})-1$. Moreover the Newton polytope of $\rho_g\rho_h$ is $N_{g}\oplus N_{h}=N_{h}$. Hence $\rho_{g}\rho_{h}$ does not satisfies Theorem \ref{results in LP} (iv), which means it is not a polytope function, particularly it can not be $\rho_{g+h}$. This leads to the necessity.

  On the other hand, when $(\rho_{h}|\rho_g)\leqslant0$, $N_{h}$ equals its $[1,n]\setminus\{i\}$-section at the origin, hence $\rho_g\rho_{h}$ is still universally indecomposable. According to \cite{LP}, $\rho_{g+h}$ is universally indecomposable and it is a summand of $\rho_{g}\rho_{h}$, so $\rho_g\rho_{h}=\rho_{g+h}$.
\end{proof}
Unistructurality conjecture claims cluster structure can be uniquely determined by the set consisting of cluster variables, which is pretty natural and direct to see through the compatibility of polytope functions since compatibility of two polytope functions $\rho_h$ and $\rho_{h\p}$ only depends on wether $\rho_h\rho_{h\p}$ equals $\rho_{h+h\p}$.

\begin{Theorem}\label{unistructurality conjecture}
  Two cluster algebras with the same cluster variables have the same cluster structure.
\end{Theorem}
\begin{proof}
  This follows from Theorem \ref{compatibility degree equals 0} and Proposition \ref{compatibility degree and multiplication}.
\end{proof}

According to Proposition \ref{compatibility degree and multiplication} and Theorem \ref{unistructurality conjecture}, we could also generalize cluster structure as maximal sets of pairwise compatible polytope functions.

\begin{Definition}
  (i)\;A polytope function $\rho_h$ is called \textbf{irreducible} if there do not exist two non-zero vector $h\p,h''\in\Z^n$ such that $\rho_h=\rho_{h\p}\rho_{h''}$.

  (ii)\;A polytope function $\rho_h$ is called \textbf{self-compatible} if $\rho_h^2=\rho_{2h}$.

  (iii)\;A \textbf{cluster} of polytope functions is a maximal set of pairwise compatible self-compatible irreducible polytope functions.
\end{Definition}

The sets of pairwise compatible self-compatible irreducible polytope functions form a complex analogous to cluster complex. Therefore, we can also construct something analogous to $g$-fan consisting of cones in $\R^n$ generated by degree vectors of a set of pairwise compatible self-compatible irreducible polytope functions and denote it by $\mathcal{C}(B)$, or simply $\mathcal{C}$ when there is no risk of confusions.

Denote by $F_h=\rho_h|_{x_i\rightarrow1,\forall i\in[1,n]}$ the $F$-polynomial of $\rho_h$ for any $h\in\Z^n$.

\begin{Lemma}\label{interior of a cone}
  Let $\{\rho_{h_i}|i\in[1,r]\}$ be a set of pairwise compatible self-compatible irreducible polytope functions, then

  (i)\;$\rho_{\sum\limits_{i=1}^r a_ih_i}=\prod\limits_{i=1}^r\rho_{h_i}^{a_i}$ for any $a_i\in\N$.

  (ii)\;for any $h\in\Z^n$ such that $\rho_h$ is a self-compatible irreducible polytope function, $\rho_h\rho_{h_i}=\rho_{h+h_i}$ for any $i\in[1,r]$ if and only if $\rho_h\rho_{\sum\limits_{i=1}^ra_ih_i}=\rho_{h+\sum\limits_{i=1}^ra_ih_i}$ for some $a_i\in\Z_{>0}$.
\end{Lemma}
\begin{proof}
  (i)\;By an induction on $\sum\limits_{i=1}^r a_i$, it is enough to show $\rho_h\rho_{h\p}\rho_{h''}=\rho_{h+h\p+h''}$ when $\rho_h$, $\rho_{h\p}$ and $\rho_{h''}$ are pairwise compatible self-compatible polytope functions.

  Since $\rho_{h}\rho_{h\p}=\rho_{h+h\p}$, it is universally indecomposable. So for any point $p\in N_{h+h\p}$ (here we regard a point with weight $co_p(N_{h+h\p})$ as $co_p(N_{h+h\p})$ different points), there is $k\in[1,n]$ such that $\rho_{h+h\p}-\hat{Y}^pX^{h+h\p}$ can not be expressed as a formal Laurent polynomial in $X_{t_k}$, where $t_k$ is connected to $t_0$ by an edge labeled $k$, in the language in \cite{LP} which means there are two sequences $0=p_0,p_1,\cdots,p_r=p$ and $i_1,\cdots,i_r$ such that $\hat{Y}^{p_s}X^{h+h\p}$ is in $m_{i_s}(p_{s-1})(\hat{y}_{i_s}+1)^{[-deg_{x_{i_s}}(\hat{Y}^{p_{s-1}}X^{h+h\p})]_+}\hat{Y}^{p_{s-1}}X^{h+h\p}$ for any $s\in[1,r]$ (here $m_k(-)$ is defined in \cite{LP} with $s_k=1$ for any $k\in[1,n]$). Moreover, we have decompositions
  $$\rho_{h}\rho_{h\p}=\sum\limits_{p\in N_h}co_{p}(N_h)\hat{Y}^pX^{h}\rho_{h\p}=\sum\limits_{p\p\in N_{h\p}}co_{p\p}(N_{h\p})\hat{Y}^{p\p}X^{h\p}\rho_{h}.$$
  This can be regarded as a variation of the above sequences: $(\hat{y}_{i_s}+1)^{[-deg_{x_{i_s}}(\hat{Y}^{p_{s-1}}X^{h+h\p})]_+}\hat{Y}^{p_{s-1}}X^{h+h\p}$ corresponds to a segment in $N_{h+h\p}$, so we can collect segments together to get polytopes. Similarly we start from $X^h\rho_{h\p}$ (or$X^{h\p}\rho_h$), in order to make the formal Laurent polynomial can be expressed as a formal Laurent polynomial in $X_{t_k}$ for any $k\in[1,n]$, we should recursively add all the other summands in the above decompositions. That is the philosophy of our construction of $\rho_h$. Then because $\rho_{h}\rho_{h''}=\rho_{h+h''}$ and $\rho_{h\p}\rho_{h''}=\rho_{h\p+h''}$, so
  $$\rho_{h}\rho_{h\p}\rho_{h''}=\sum\limits_{p\in N_h}co_{p}(N_h)\hat{Y}^pX^{h}\rho_{h\p+h''}=\sum\limits_{p\p\in N_{h\p}}co_{p\p}(N_{h\p})\hat{Y}^{p\p}X^{h\p}\rho_{h+h''}.$$
  Since multiplying $\rho_{h''}$ does not affect the form of above two decompositions, for any point $p\in N_{h+h\p+h''}$, there is $k\in[1,n]$ such that $\rho_{h+h\p+h''}-\hat{Y}^pX^{h+h\p+h''}$ can not be expressed as a formal Laurent polynomial in $X_{t_k}$, that is, $\rho_{h}\rho_{h\p}\rho_{h''}$ is universally indecomposable.

  By Corollary 6.7 %??
  in \cite{LP}, $\rho_{h+h\p+h''}$ is a universally indecomposable summand of $\rho_{h}\rho_{h\p}\rho_{h''}$, hence  $\rho_{h}\rho_{h\p}\rho_{h''}=\rho_{h+h\p+h''}$, which completes the proof of (i).

  (ii)\;Because $\rho_h\rho_{h_i}=\rho_{h+h_i}$ for any $i\in[1,r]$, so $\{\rho_h\}\cup\{\rho_{h_i}|i\in[1,r]\}$ is a set of pairwise compatible self-compatible irreducible polytope functions. By (i), $\rho_h\rho_{\sum\limits_{i=1}^ra_ih_i}=\rho_{h+\sum\limits_{i=1}^ra_ih_i}$ for any $a_i\in\N$, which induces the necessity.

  Assume $\rho_h\rho_{\sum\limits_{i=1}^ra_ih_i}=\rho_{h+\sum\limits_{i=1}^ra_ih_i}$ for some $a_i\in\Z_{>0}$ but $\rho_h\rho_{h_j}\neq\rho_{h+h_j}$ for some $j\in[1,r]$. According to \cite{LP}, the structure constants of $\hat{\mathcal{P}}$ are positive, let $\rho_h\rho_{h_j}=\rho_{h+h_j}+\sum\limits_{h'}c_{h'}\hat{Y}^{h+h_j-h'}\rho_{h'}$ with $c_{h'}\in\Z_{>0}$. Then again by (i), $\rho_{\sum\limits_{i=1}^r a_ih_i}=\prod\limits_{i=1}^r\rho_{h_i}^{a_i}$, so $\rho_h\rho_{\sum\limits_{i=1}^ra_ih_i}=\rho_h\rho_{h_j}^{a_j}\prod\limits_{i\neq j}\rho_{h_i}^{a_i}$, which contains a proper summand $\rho_{h+h_j}\rho_{h_j}^{a_j-1}\prod\limits_{i\neq j}\rho_{h_i}^{a_i}$ as the structure constants are positive. While $\rho_{h+\sum\limits_{i=1}^ra_ih_i}$ is a summand of $\rho_{h+h_j}\rho_{h_j}^{a_j-1}\prod\limits_{i\neq j}\rho_{h_i}^{a_i}$, hence $\rho_h\rho_{\sum\limits_{i=1}^ra_ih_i}$ does not equal to $\rho_{h+\sum\limits_{i=1}^ra_ih_i}$, which leads to the efficiency.
\end{proof}

\begin{Proposition}\label{sign-coherence of G-matrices}
  If $\rho_h$ and $\rho_{h\p}$ are compatible with $h=(h_1,\cdots,h_n)^\top$ and $h=(h\p_1,\cdots,h\p_n)^\top$, then $h_ih\p_i\in\N$ for any $i\in[1,n]$. In particular, $G$-matrices are row sign-coherent.
\end{Proposition}
\begin{proof}
  Assume on the contrary we have $h_{k}h\p_{k}\in\Z_{<0}$ for some $k\in[1,n]$. With out loss of generality we may assume $h_{k}>0$ while $h\p_k<0$. Then we can find $a\in\Z_{>0}$ such that $ah_{k}+h\p_{k}\geqslant0$.

  Since $\rho_h$ and $\rho_{h\p}$ are compatible, following from Lemma \ref{interior of a cone} (i), we have $\rho_{h}^{a}\rho_{h\p}=\rho_{ah+h\p}$ and hence $N_{h}^{\oplus a}\oplus N_{h\p}=N_{ah+h\p}$.

  Because $h\p_{k}<0$, so $co_{e_{k}}(N_{h\p})>0$ (otherwise $L^{t_k}(\rho_{h\p})$ is not a Laurent polynomial with non-negative coefficients, where $t_k$ is connected to $t_0$ by an edge labeled $k$). Moreover due to the construction of polytope functions, we have $co_{0}(N_{h})>0$ and $co_{p}(N_{h})\in\N$ for any $p,h\in\Z^n$. So $co_{e_{k}}(N_{h}^{\oplus a}\oplus N_{h\p})>0$. However, according to Theorem \ref{results in LP} (iii) and (iv), $co_{e_{k}}(N_{ah+h\p})=0$ because $ah_{k}+h\p_{k}\geqslant0$, which induces $N_{h}^{\oplus a}\oplus N_{h\p}\neq N_{ah+h\p}$ and hence leads to a contradiction.
\end{proof}

\begin{Lemma}\label{intersection of cones}
  Let $\{\rho_{h}|h\in V\}$ and $\{\rho_{h\p}|h\p\in V\p\}$ be two sets of pairwise compatible self-compatible irreducible polytope functions, then $\text{cone}(V)\cap \text{cone}(V\p)=\text{cone}(V\cap V\p)$.
\end{Lemma}
\begin{proof}
  Obviously $\text{cone}(V)\cap \text{cone}(V\p)\supseteq \text{cone}(V\cap V\p)$. Choose an arbitrary vector $f\in \text{cone}(V)\cap \text{cone}(V\p)$, according to Lemma \ref{interior of a cone} (i), $\rho_f=\prod\limits_{h\in V}\rho_{h}^{a_{h}}=\prod\limits_{h\p\in V\p}\rho_{h\p}^{a\p_{h\p}}$ for $a_h,a\p_{h\p}\in\N$ such that $f=\sum\limits_{h\in V}a_h h=\sum\limits_{h\p\in V\p}a\p_{h\p} h\p$. Then $F_f=\prod\limits_{h\in V}F_{h}^{a_{h}}=\prod\limits_{h\p\in V\p}F_{h\p}^{a\p_{h\p}}$. Since $\Q[[Y]]$ is a unique factorization domain, $F_h$ is irreducible and $F_h$ has constant 1 for any $h\in\Z^n$ according to \cite{LP}, so $\{F_h|h\in V,a_h\neq0\}=\{F_{h\p}|h\p\in V\p,a\p_{h\p}\neq0\}$ and $a_{h}=a\p_{h\p}$ when $F_h=F_{h\p}$.

  On the other hand, because of Theorem \ref{results in LP} (iv), $d_k(\rho_h)$ equals the maximal length of edges of $N_h$ parallel to $e_k$ for any $k\in[1,n]$ when $\rho_h$ is irreducible. Hence in this case $\rho_h=X^\alpha F_h(\hat{Y})$, where $\alpha\in\Z^n$ is chosen to make $d_k(X^\alpha F_h(\hat{Y}))$ equals the maximal length of edges of $N_h$ parallel to $e_k$ for any $k\in[1,n]$, which means $\rho_h$ is uniquely determined by $F_h$.

  Therefore, $\{h\in V|a_h\neq0\}=\{h\p\in V\p|a\p_{h\p}\neq0\}$ and thus $f\in \text{cone}(V\cap V\p)$. So $\text{cone}(V)\cap \text{cone}(V\p)\subseteq \text{cone}(V\cap V\p)$.
\end{proof}
\begin{Lemma}\label{two vectors in one line}
  Assume $\rho_h$ and $\rho_{h'}$ are two different self-compatible irreducible polytope functions, then

  (i)\;$\R_{>0}h\neq\R_{>0}h'$.

  (ii)\;$\R h\neq\R h'$ if $\rho_h$ and $\rho_{h'}$ are compatible.
\end{Lemma}
\begin{proof}
  (i)\; Assume $\R_{>0}h=\R_{>0}h'$, then there must be a primitive vector $g$ and two different positive integers $a$ and $b$ such that $h=ag$ and $h'=bg$. Then $\rho_{abg}=\rho_{h}^b=\rho_{h'}^a$, which induces $F_{abg}=F_{h}^b=F_{h'}^a$. Since $\Q[[Y]]$ is a unique factorization domain, there is $f(Y)\in\Q[[Y]]$ such that $F_h=f(Y)^a$ and $F_{h'}=f(Y)^b$. Hence $f(Y)\in\Z[[Y]]$. As explained in the proof of Lemma \ref{intersection of cones}, $\rho_h$ is uniquely determined by $F_h$, so $\rho_{h}=(X^g f(\hat{Y}))^a$. Since $\rho_{h}$ is a universally positive summand of $\rho_g^{a}$ and $\rho_g$ is universally indecomposable, $\rho_g=X^g f(\hat{Y})$, which contradicts $\rho_h$ being irreducible.

  (ii)\;By (i), only need to prove $\R_{>0}h\neq\R_{<0}h'$ if $\rho_h$ and $\rho_{h'}$ are compatible. Assume $\R_{>0}h=\R_{<0}h'$, then there must be a primitive vector $g$ and two different positive integers $a$ and $b$ such that $h=ag$ and $h'=-bg$. Because $\rho_h$ and $\rho_{h'}$ are compatible and self-compatible, so $1=\rho_0=\rho_{bh+ah'}=\rho_h^b\rho_{h'}^a$ by Lemma \ref{interior of a cone} (i), which is impossible as $\rho_h$ and $\rho_{h'}$ are universally positive.
\end{proof}
Combining Lemma \ref{interior of a cone}, Lemma \ref{intersection of cones} and Lemma \ref{two vectors in one line}, we could see more similarities between $\mathcal{C}$ and $g$-fan through the following results.
\begin{Proposition}\label{dimension of a cone}
  Let $\{\rho_{h}|h\in V\}$ be a set of pairwise compatible self-compatible irreducible polytope functions, then $dim(\text{cone}(V))=|V|$.
\end{Proposition}
\begin{proof}
  This can be done by an induction on $|V|$. It is true when $|V|\leqslant 2$ due to Lemma \ref{two vectors in one line} (ii). Now assume $|V|\geqslant2$. Let $\rho_{h\p}$ be a self-compatible irreducible polytope function which is compatible with $\rho_h$ for any $h\in V$ and $h\p\notin V$. According to inductive assumption, $dim(\text{cone}(V\setminus\{h\}\cup\{h\p\}))=|V|$ for any $h\in V$, so $h\p$ does not lie in $\text{aff}(V\setminus\{h\})$. If $h\p\in\text{aff}(V)$, we can find $h\in V$ such that $h\p$ and $h$ lie in the same half-space of $\text{aff}(V)$ with respect to $\text{aff}(V\setminus\{h\})$. Then due to Lemma \ref{intersection of cones}, $h=h\p$, which contradicts our assumption. Therefore, $h\p\notin\text{aff}(V)$, which means $dim(\text{cone}(V\cup\{h\p\}))=dim(\text{cone}(V))+1=|V|+1$.
\end{proof}
Therefore, there are at most $n$ polytope functions in a cluster.
\begin{Proposition}
  $\mathcal{C}$ is a fan and it contains all cones in $g$-fan.
\end{Proposition}
\begin{proof}
  For any cone in $\mathcal{C}$ generated by a set $S$ of degree vectors of pairwise compatible self-compatible irreducible polytope functions, any face of it is generated by a subset of $S$, which is a cone in $\mathcal{C}$. By Lemma \ref{intersection of cones}, the intersection of two cones is their common face. Therefore, $\mathcal{C}$ is a fan.

  Then by Theorem \ref{compatibility degree equals 0} and Proposition \ref{compatibility degree and multiplication}, cones in $g$-fan are also cones in $\mathcal{C}$.
\end{proof}

\section{Duality of $G$-matrices and $C$-matrices}
Another way to look at clusters can be given via polytope realization of $G$-matrices or $C$-matrices as they one-to-one correspond to clusters. Before introducing this realization, we would like to generalize the duality of $G$-matrices and $C$-matrices presented in \cite{NZ} to TSSS case in this section.

\begin{Theorem}\label{GB=BC}
  \cite{FZ4}For a TSSS cluster algebra $\A$ and any $t,t\p\in\T_n$, $G^{B_t;t}_{t\p}B_{t\p}=B_tC^{B_t;t}_{t\p}$.
\end{Theorem}

\begin{Definition}
  Let $B=(b_{ij})$ and $B\p=(b\p_{ij})$ be two $n_1\times n_2$ matrices.

  (i)\;We say they are \textbf{weakly sign-synchronic} to each other if $b_{ij}b\p_{ij}\geqslant0$ for any $i\in[1,n_1]$ and $j\in[1,n_2]$.

  (ii)\;We say they are \textbf{sign-synchronic} to each other if either $b_{ij}b\p_{ij}>0$ or $b_{ij}=b\p_{ij}=0$ for any $i\in[1,n_1]$ and $j\in[1,n_2]$.
\end{Definition}

The following result confirms conjecture 3.13 proposed by Siyang Liu in \cite{L}, which is also proved in the same paper to hold for an acyclic sign-skew-symmetric cluster algebra.

\begin{Proposition}\label{sign-synchronicity of c-vectors for B and -B^T}
  Let $B$ be a TSSS matrix, then for any $t\in\T_n$, $G_t^{B;t_0}$ is sign-synchronic to $G_t^{-B^\top;t_0}$ and $C_t^{B;t_0}$ is sign-synchronic to $C_t^{-B^\top;t_0}$.
\end{Proposition}
We will leave the proof of Proposition \ref{sign-synchronicity of c-vectors for B and -B^T} to next section. Before that, we would like to present the duality between $G$-matrices and $C$-matrices.

\begin{Theorem}\label{relation between G-matrix and C-matrix}
  For any $t,t\p\in\T_{n}$, the following equations hold

  (i)\;$(G^{B_{t\p};t\p}_{t})^{\top}=C^{B_{t}^{\top};t}_{t\p}$.

  (ii)\;$G^{B_t;t}_{t\p}G^{-B_{t\p};t\p}_t=I_n$.

  (iii)\;$G^{B_t;t}_{t\p}(C^{-B_{t}^{\top};t}_{t\p})^{\top}=I_n$ and $C^{B_t;t}_{t\p}C^{-B_{t\p};t\p}_t=I_n$.
\end{Theorem}
\begin{proof}
  Fix $t$, we prove the theorem by induction on the length of path from $t$ to $t\p$ in $\T_{n}$. When $t\p=t$, $G^{B_{t\p};t\p}_{t}=C^{B_t^{\top};t}_{t\p}=G^{B_t;t}_{t\p}=G^{-B_{t\p};t\p}_t=I_{n}$, so both equalities hold.

  Now assume $(G^{B_{t\p};t\p}_{t})^{\top}=C^{B_t^{\top};t}_{t\p}$ and $G^{B_t;t}_{t\p}G^{-B_{t\p};t\p}_t=I_n$, then we need to prove $(G^{B_{t\p_k};t\p_{k}}_{t})^{\top}=C^{B_t^{\top};t}_{t\p_{k}}$ and $G^{B_t;t}_{t\p_k}G^{-B_{t\p_k};t\p_k}_t=I_n$ for any $k\in[1,n]$, where $t\p_{k}$ is connected to $t\p$ by an edge labeled $k$.

  (i)\;Let $C^{B_t^{\top};t}_{t\p}=(c_{ij}^{t\p})_{i,j\in[1,n]}$, $C^{B_t^{\top};t}_{t\p_{k}}=(c_{ij}^{t\p_k})_{i,j\in[1,n]}$, $G^{B_{t\p};t\p}_{t}=(g_{ij}^{t\p})_{i,j\in[1,n]}$, $G^{B_{t\p_k};t\p_{k}}_{t}=(g_{ij}^{t\p_{k}})_{i,j\in[1,n]}$ and $B_{t\p}=(b_{ij}^{t\p})_{i,j\in[1,n]}$.

  According to the definition of $C$-matrices, $\begin{pmatrix}
     B^{B_{t}^{\top};t}_{t\p_{k}} \\
     C^{B_{t}^{\top};t}_{t\p_{k}}
   \end{pmatrix}=\mu_{k}(\begin{pmatrix}
                           B^{B_t^{\top};t}_{t\p} \\
                           C^{B_t^{\top};t}_{t\p}
                         \end{pmatrix})$. So following (\ref{equation: mutation of B}),
  \begin{equation*}
    c_{ij}^{t\p_{k}}=\left\{\begin{array}{ll}
                       -c_{ij}^{t\p}, & \text{if }j=k; \\
                       c_{ij}^{t\p}+sign(b_{kj}^{t\p})[c_{ik}^{t\p}b_{kj}^{t\p}]_+, & \text{otherwise}.
                     \end{array}\right.
  \end{equation*}
  On the other hand, according to Theorem \ref{results in LP} (i) and (ii),
  \begin{equation*}
    g_{ji}^{t\p_{k}}=\left\{\begin{array}{ll}
                       -g_{ji}^{t\p}, & \text{if }j=k; \\
                       g_{ji}^{t\p}+sign(b_{kj}^{t\p})[g_{ki}^{t\p}b_{kj}^{t\p}]_+, & \text{otherwise}.
                     \end{array}\right.
  \end{equation*}
  According to the inductive assumption, $c_{ij}^{t\p}=g_{ji}^{t\p}$ for any $i,j\in[1,n]$. Then by comparing the above two formulas we get $c_{ij}^{t\p_k}=g_{ji}^{t\p_k}$ for any $i,j\in[1,n]$, which completes the proof.

  (ii)\;By the exactly same discussion as that in \cite{NZ} for skew-symmetrizable case or that in \cite{L} for acyclic sign-skew-symmetric case, we can see (ii) holds so long as Proposition \ref{sign-synchronicity of c-vectors for B and -B^T} holds. For the completeness of proof, we repeat their discussion here.

  Choose arbitrary $r,s\in[1,n]$. For simplicity denote by $g_{i}$ and $g\p_i$ the entries of $G^{B_{t};t}_{t\p}$ and $G^{B_{t};t}_{t\p_k}$ respectively lying in the intersection of the $r$-th row and the $i$-th column while by $\bar{g}_{i}$ and $\bar{g}\p_i$ the entries of $G^{-B_{t\p};t\p}_t$ and $G^{-B_{t\p_k};t\p_k}_{t}$ respectively lying in the intersection of the $i$-th row and the $s$-th column.

  According to (\ref{equation: mutation of g-vectors}) and Theorem \ref{results in LP} (i),
  \begin{equation*}
    g\p_i=\left\{\begin{array}{ll}
                   g_i, & \text{if }i\neq k; \\
                   -g_k+\sum\limits_{j=1}^{n}[-b_{jk}^{t\p}]_+g_j-\sum\limits_{j=1}^{n}[-c_{jk}^{t\p}]_+b_{rj}^t, & \text{if }i=k.
                 \end{array}\right.
  \end{equation*}
  while
  \begin{equation*}
    \bar{g}\p_i=\left\{\begin{array}{ll}
                   \bar{g}_i+[-b_{ik}^{t\p}]_+\bar{g}_k-b_{ik}^{t\p}[-\bar{g}_k]_+, & \text{if }i\neq k; \\
                   -\bar{g}_k, & \text{if }i=k.
                 \end{array}\right.,
  \end{equation*}
  where $b_{ij}^t$ represents the corresponding entry of $B_t$.

  Therefore,
  \begin{equation*}
    \begin{array}{ll}
      \sum\limits_{i=1}^{n}g\p_i\bar{g}\p_i & =\sum\limits_{i\neq k}g\p_i\bar{g}\p_i+\sum\limits_{i=k}g\p_i\bar{g}\p_i \\
       & =\sum\limits_{i\neq k}g_i(\bar{g}_i+[-b_{ik}^{t\p}]_+\bar{g}_k-b_{ik}^{t\p}[-\bar{g}_k]_+) +\bar{g}_k(g_k-\sum\limits_{j=1}^{n}[-b_{jk}^{t\p}]_+g_j+\sum\limits_{j=1}^{n}[-c_{jk}^{t\p}]_+b_{rj}^t) \\
       & = \sum\limits_{i=1}^{n}g_i\bar{g}_i-\sum\limits_{i=1}^{n}g_ib_{ik}^{t\p}[-\bar{g}_k]_++\sum\limits_{j=1}^{n}b_{rj}^t[-c_{jk}^{t\p}]_+\bar{g}_k
    \end{array}
  \end{equation*}
  By (i) and Proposition \ref{sign-synchronicity of c-vectors for B and -B^T}, $\bar{g}_kc_{jk}^{t\p}\geqslant0$. And by Theorem \ref{GB=BC}, $\sum\limits_{i=1}^{n}g_ib_{ik}^{t\p}-\sum\limits_{j=1}^{n}b_{rj}^tc_{jk}^{t\p}=0$. Therefore, $\sum\limits_{i=1}^{n}g\p_i\bar{g}\p_i=\sum\limits_{i=1}^{n}g_i\bar{g}_i$ and hence $G^{B_t;t}_{t\p_k}G^{-B_{t\p_k};t\p_k}_t=G^{B_t;t}_{t\p}G^{-B_{t\p};t\p}_t=I_n$.

  (iii) directly follows from (i) and (ii).
\end{proof}

\begin{Remark}
  By Theorem \ref{relation between G-matrix and C-matrix} (i), the row sign-coherence of $G$-matrices in Proposition \ref{sign-coherence of G-matrices} induces the column sign-coherence of $C$-matrices. This can also be seen from \cite{LP} where we proved that every $F$-polynomial has constant term 1 and a unique maximal term with constant coefficient 1, which is equivalent to the column sign-coherence of $C$-matrices as showed in \cite{FZ4}. And since a $c$-vector can never be 0, following \cite{NZ}, we denote $sign(c_{i;t})=1$ if $c_{i;t}\in\N^n$ and $sign(c_{i;t})=-1$ otherwise.
\end{Remark}

\begin{Corollary}\label{C-matrices determine seeds}
  Let $\A$ be a cluster algebra and $t,t\p\in\T_n$, then up to a relabeling $C_t=C_t\p$ if and only if $\Sigma_{t}=\Sigma_{t\p}$.
\end{Corollary}
\begin{proof}
  It is equivalent to $X_{t}=X_{t\p}$ up to the same relabeling by Theorem \ref{relation between G-matrix and C-matrix} and Theorem 5.13 %??
  in \cite{LP} claiming that $x_{i;t}=x_{i;t\p}$ if and only if $g_{i;t}=g_{i;t\p}$ for any $i\in[1,n]$ and $t,t\p\in\T_n$, which is equivalent to $\Sigma_{t}=\Sigma_{t\p}$ up to the same relabeling due to Theorem \ref{a straight way}.
\end{proof}

\begin{Corollary}
   Let $\A(B)$ and $\A(-B^\top)$ be two TSSS cluster algebra, their exchange graphes $E(\A(B))$ and $E(\A(-B^\top))$ are the same.
\end{Corollary}
\begin{proof}
  This follows from the facts that Both $E(\A(B))$ and $E(\A(-B^\top))$ are quotients of the $n$-regular tree and their fundamental groups are the same due to Theorem \ref{relation between G-matrix and C-matrix} (iii) and Corollary \ref{C-matrices determine seeds}.
\end{proof}

\begin{Corollary}
  Let $\A$ be a cluster algebra associated to $B$ and $t\in\T_n$, then $G_t^{B;t_0}\in Mat_{n\times n}(\Z_{\leqslant0})$ if and only if $G_t^{B;t_0}=-I_n$. And in this case, $B_t=B$.
\end{Corollary}
\begin{proof}
  The proof is the same as skew-symmetrizable case. We repeat it here for completeness.

  The sufficiency is obvious. Assume $G_t^{B;t_0}\in Mat_{n\times n}(\Z_{\leqslant0})$, we need to show it can only be $-I_n$. Because of the column sign-coherence of $c$-vectors and Theorem \ref{relation between G-matrix and C-matrix} (iii), so $C_t^{-B^\top;t_0}\in Mat_{n\times n}(\Z_{\leqslant0})$. Therefore the only possible case to satisfy the above conditions is that $c_{i;t}^{-B^\top;t_0}=-e_j$ for some $j\in[1,n]$ and the cone generated by $\{g_{k;t}^{B;t_0}|k\neq i\}$ lies in the hyper plane $z_j=0$ for any $i\in[1,n]$, that is, $G_t^{B;t_0}=C_t^{-B^\top;t_0}=-I_n$. Then by Proposition \ref{sign-synchronicity of c-vectors for B and -B^T}, we also have $C_t^{B;t_0}\in Mat_{n\times n}(\Z_{\leqslant0})$, which leads to $C_t^{B;t_0}=-I_n$ similarly.

  Suppose there is a path $t_0\overset{i_1}-t_1\overset{i_2}-\cdots \overset{i_s}-t_s=t$ in $\T_n$. According to \cite{BFZ}, \cite{NZ} and the column sign-coherency of $C$-matrices, we can rewrite the mutation formulas of exchange matrices, $G$-matrices and $C$-matrices respectively as following
  \[B_t=E_{i_s;t_{s-1}}\cdots E_{i_1;t_0}BF_{i_1;t_0}\cdots F_{i_s;t_{s-1}},\]
  \begin{equation}\label{equation: mutation of G and C in matrix form}
    G_t^{B;t_0}=G_{t_0}^{B;t_0}E_{i_1;t_0}\cdots E_{i_s;t_{s-1}}\quad\text{ and }\quad C_t^{B;t_0}=C_{t_0}^{B;t_0}F_{i_1;t_0}\cdots F_{i_s;t_{s-1}},
  \end{equation}
  where $E_{k;t_r}=(e_{ij}^{k;t_r})$ and $F_{k;t_r}=(f_{ij}^{k;t_r})$ satisfies
  \begin{equation*}
    e_{ij}^{k;t_r}=\left\{\begin{array}{ll}
                            \delta_{ij}, & \text{if }j\neq k; \\
                            -1, & \text{if }i=j=k; \\
                            {[-sign(c_{k;t}^{B;t_0})b_{ik}^{t_r}]_+}, & \text{if }i\neq j=k.
                          \end{array}\right.
  \end{equation*}
  and
  \begin{equation*}
    f_{ij}^{k;t_r}=\left\{\begin{array}{ll}
                            \delta_{ij}, & \text{if }i\neq k; \\
                            -1, & \text{if }i=j=k; \\
                            {[sign(c_{k;t}^{B;t_0})b_{kj}^{t_r}]_+}, & \text{if }i=k\neq j.
                          \end{array}\right.
  \end{equation*}
  for any $k\in[1,n]$ and $r\in[0,s-1]$.
  Therefore,
  \[E_{i_1;t_0}\cdots E_{i_s;t_{s-1}}=-I_n\quad\text{ and }\quad F_{i_1;t_0}\cdots F_{i_s;t_{s-1}}=-I_n,\]
  which induces
  \[E_{i_s;t_{s-1}}\cdots E_{i_1;t_0}=-I_n.\]
  Hence
  \[B_t=E_{i_s;t_{s-1}}\cdots E_{i_1;t_0}BF_{i_1;t_0}\cdots F_{i_s;t_{s-1}}=-I_n B(-I_n)=B.\]
\end{proof}

\begin{Corollary}\label{c-vectors or g-vectors form a Z-basis}
  The column vectors or the transpose of row vectors of a $C$-matrix or a $G$-matrix form a $\Z$-basis of $\Z^n$.
\end{Corollary}
\begin{proof}
  According to the column sign-coherency of $C$-matrices and \cite{NZ}, the mutation formulas of $G$-matrices and $C$-matrices can be rewrite as (\ref{equation: mutation of G and C in matrix form}) via multiplying invertible integer matrices on the right hand-side. And since $C_{t_0}^{B;t_0}=G_{t_0}^{B;t_0}=I_n$, the column vectors of a $C$-matrix or $G$-matrix form a $\Z$-basis of $\Z^n$. Then the rest part follows from Theorem \ref{relation between G-matrix and C-matrix} (i).
\end{proof}

\section{Polytope Realization of $G$-matrices and $C$-matrices}
Based on the similarity between the mutation behavior of polytopes introduced in Section 3 and that of $C$-matrices under mutations of the initial seed, we want to figure out the concrete relation between them, or in other words, the realization of $g$-vectors (or $c$-vectors) in our polytopes corresponding to polytope functions.

In this section, we will show that we can choose a particular edge set of a special polytope such that primitive vectors of these edges form a $C$-matrix and moreover it is compatible with mutations. With this find we are able to realize all $G$-matrices and $C$-matrices in the mutation classes of the above polytope, which allows us to construct a complete fan containing $g$-fan.

In the sequel, denote $supp(u)=\{i\in[1,n]|u_i\neq0\}$ for a vector $u=(u_1,\cdots,u_n)\in\R^n$. For a matrix $B$ and its sub-index set $I$, denote by $B_I$ the matrix obtained from $B$ by deleting rows and columns parameterized by indices in $I$ and by $B[I]$ the matrix obtained from $B$ by deleting columns parameterized by indices not in $I$.

\begin{Definition}
  Let $S$ be a face of a polytope $N$ in $\R^n$ and $i\in[1,n]$,

  (i)\;$S$ is called a \textbf{top face} of $N$ with respect to $i$ if there is a vector $v\in\R^n$ such that the $i$-th entry of $v$ is positive and $S\subseteq\{p\in\R^n|v^\top p=0\}$, $N\nsubseteq\{p\in\R^n|v^\top p=0\}$ and $N\subseteq\{p\in\R^n|v^\top p\leqslant0\}$.

  (ii)\;$S$ is called a \textbf{bottom face} with respect to $i$ if there is a vector $v\in\R^n$ such that the $i$-th entry of $v$ is negative and $S\subseteq\{p\in\R^n|v^\top p=0\}$, $N\nsubseteq\{p\in\R^n|v^\top p=0\}$ and $N\subseteq\{p\in\R^n|v^\top p\leqslant0\}$.
\end{Definition}

In the rest of this section, we denote by $N$ the polytope corresponding to $\prod\limits_{i=1}^n\rho_{e_i}\rho_{-e_i+[-b_i]_+}=\prod\limits_{i=1}^nM_i$.
For any $j\in[1,n]$, let $q^+_j$ be the point $\sum\limits_{i=1}^n e_i$, $p^+_j$ be the point $\sum\limits_{i\neq j}^n e_i$, $q^-_j$ be the point $e_j$ and $p^-_j$ be the origin. Let $L^\epsilon=\{l^\epsilon_1,\cdots,l^\epsilon_n\}$ for $\epsilon\in\{+,-\}$, where $l^\epsilon_j=\overline{p^\epsilon_jq^\epsilon_j}$. For a mutation sequence $\mu$, denote by $\mu(L^\epsilon)=\{\mu(l^\epsilon_1),\cdots,\mu(l^\epsilon_n)\}$ the subset of $E(\mu(N))$ such that $\mu(l^\epsilon_i)$ correlating to $l^\epsilon_i$ under $\mu$.

Note that according to Theorem \ref{results in LP} (iv), Theorem \ref{compatibility degree} and Proposition \ref{compatibility degree equals 0}, $dim(\mu(N))=n$. And moreover due to the following correspondence of outer normal vectors and $g$-vectors in Theorem \ref{c-vectors in N} (ii) and row sign-coherence of $G$-matrices in Proposition \ref{sign-coherence of G-matrices}, there will not be indices $i,j,k\in[1,n]$ such that the facet of $\mu(N)$ determined by $\mu(L^\epsilon)\setminus\{\mu(l^\epsilon_i)\}$ is a top facet with respect to $k$ while the facet of $\mu(N)$ determined by $\mu(L^\epsilon)\setminus\{\mu(l^\epsilon_j)\}$ is a bottom facet with respect to $k$. Therefore, there will always be exactly one edge of $\mu_k\mu(N)$ correlating to $\mu(l^\epsilon_i)$ under $\mu_k$ for any $i\in[1,n]$, and thus $\mu(L^\epsilon)$ is uniquely determined for each $\mu$. Moreover, Corollary \ref{c-vectors or g-vectors form a Z-basis} and Theorem \ref{c-vectors in N} (i) ensure $ldim(S)=dim(S)=dim(S\p)=ldim(S\p)$ for any face $S$ of $N$ determined by a subset of $L^{\epsilon}$ and the face $S'$ of $\mu(N)$ correlating to $S$ under an arbitrary mutation sequence $\mu$. Therefore, the polytope projections in Theorem \ref{mutation of faces} are always isomorphisms here for the face $S'$ correlating to $S$ under an arbitrary mutation sequence $\mu$.

It should be more rigorous to include the above explanation in the induction of the proof below. However, we prefer to claim prior the well-definedness of $\mu(l^\epsilon_i)$ and use isomorphisms rather than polytope projections before the proof to simplify notations.

Let $\mu(p^\epsilon_j)$ and $\mu(q^\epsilon_j)$ be the vertices of $\mu(l^\epsilon_j)$ correlate respectively to $p^\epsilon_j$ and $q^\epsilon_j$ under $\mu$. For the same reason as above, these vertices are uniquely determined. Denote $\overrightarrow{\mu(L^\epsilon)}=(\overrightarrow{\mu(l^\epsilon_1)},\cdots,\overrightarrow{\mu(l^\epsilon_n)})$, where \[\overrightarrow{\mu(l^\epsilon_i)}=\frac{(\mu(q^\epsilon_i)-\mu(p^\epsilon_i))}{gcd(\mu(q^\epsilon_i)-\mu(p^\epsilon_i))},\]
with $gcd(\mu(q^\epsilon_i)-\mu(p^\epsilon_i))$ represents the (positive) greatest common divisor of entries in $\mu(q^\epsilon_i)-\mu(p^\epsilon_i)$.

Due to Theorem \ref{results in LP} (iii), $\overrightarrow{\mu(l^\epsilon_i)}$ is sign-coherent and non-zero for any $i\in[1,n]$. So similar to $c$-vectors, we denote $sign(\overrightarrow{\mu(l^\epsilon_i)})=1$ if $\overrightarrow{\mu(l^\epsilon_i)}\in\N^n$ and otherwise $sign(\overrightarrow{\mu(l^\epsilon_i)})=-1$.

When $n\geqslant2$, denote by $\overrightarrow{\mu(l^\epsilon_i)}^\bot$ the primitive outer normal vector of the facet of $\mu(N)$ determined by $\mu(L^\epsilon)\setminus\{\mu(l^\epsilon_i)\}$ and $\overrightarrow{\mu(L^\epsilon)}^\bot=(\overrightarrow{\mu(l^\epsilon_1)}^\bot,\cdots,\overrightarrow{\mu(l^\epsilon_n)}^\bot)$.

\begin{Lemma}\label{positive vector and top facet}
  In the above settings, for any $i\in[1,n]$ and $n\geqslant2$,

  (i)\;$\mu(l^\epsilon_i)$ is a top edge of $N$ with respect to $i$ if and only if there is $j\neq i\in[1,n]$ such that the $i$-th entry of $\overrightarrow{\mu(l^\epsilon_1)}^\bot$ is positive.

  (ii)\;$\mu(l^\epsilon_i)$ is a bottom edge of $N$ with respect to $i$ if and only if there is $j\neq i\in[1,n]$ such that the $i$-th entry of $\overrightarrow{\mu(l^\epsilon_1)}^\bot$ is negative.
\end{Lemma}
\begin{proof}
  The Lemma follows from the definition of top or bottom face, the definition of outer normals and the fact that $dim(\mu(N))=n$.
\end{proof}

Given any subset $I\subseteq[1,n]$, let $N_{[1,n]\setminus I}$ be the polytope corresponding to $\prod\limits_{i\in I}M_i|_{\A(B_{[1,n]\setminus I})}$. It can be embedded in $\R^n$ via $\gamma_{[1,n]\setminus I;0}$ when $\epsilon=-$ or via $\gamma_{[1,n]\setminus I;\sum\limits_{i\notin I}e_i}$ when $\epsilon=+$. According to the discussion in Section 3, any mutation sequence $\mu=\mu_{i_r}\cdots\mu_{i_1}$ in $\A$ induces a mutation sequence $\tilde{\mu}^\epsilon=\mu^\epsilon_{j_r}\cdots\mu^\epsilon_{j_1}$ in $\A(B_{[1,n]\setminus I})$ such that $\mu^\epsilon_{j_k}$ is either a single mutation or $\emptyset$ for any $k\in[1,r]$ and a non-negative polytope isomorphism $\tau^\epsilon_s$ of $\mu_{j_s}^\epsilon\cdots\mu_{j_1}^\epsilon(N_{[1,n]\setminus I})|_{\A(B_{[1,n]\setminus I})}$ and the $|I|$-dimensional face of $\mu_{i_s}\cdots\mu_{i_1}(N)$ determined by $\{\mu_{i_s}\cdots\mu_{i_1}(l^\epsilon_i)|i\in I\}$ for any $s\in[0,r]$. In this special case, as the $|I|$-dimensional face of $\mu_{i_s}\cdots\mu_{i_1}(N)$ determined by $\{\mu_{i_s}\cdots\mu_{i_1}(l^\epsilon_i)|i\in I\}$ is unique, $\tilde{\mu}^\epsilon$ is uniquely determined by $\mu$ and $\epsilon$.

\begin{Theorem}\label{c-vectors in N}
  Let $\A$ be a TSSS cluster algebra and $\mu=\mu_{i_r}\cdots\mu_{i_1}$ be a mutation sequence from $\Sigma_{t_0}$ to $\Sigma_t$.

  (i)\;$C_{t_0}^{\epsilon B_t;t}=\overrightarrow{\mu(L^\epsilon)}$ for $\epsilon\in\{+,-\}$.

  (ii)\;$G_{t_0}^{-\epsilon B_t^\top;t}=\epsilon\overrightarrow{\mu(L^\epsilon)}^\bot$ for $\epsilon\in\{+,-\}$ when $n\geqslant2$.

  (iii)\;$\mu^\epsilon_{j_s}\neq\emptyset$ if and only if $i_s\notin \bigcup\limits_{i\in [1,n]\setminus I}supp(g_{i;t_0}^{-\epsilon B_{t(s-1)}^\top;t(s-1)})$ for any $s\in[1,r]$, where $t(s-1)$ is connected to $t_0$ by the path corresponding to $\mu_{i_{s-1}}\cdots\mu_{i_1}$. In this case, there is a unique index $j\in I$ satisfying
  $$\pi_{[1,n]\setminus I}((G_{t_0}^{-\epsilon B_{t(s-1)}^\top;t(s-1)})^\top[i_s])=(G_{t_0}^{-\epsilon (B_{[1,n]\setminus I})_{t\p(s-1)}^\top;t\p(s-1)})^\top[j],$$
  where $t\p(s-1)$ is connected to $t_0$ by the path corresponding to $\mu^\epsilon_{j_{s-1}}\cdots\mu^\epsilon_{j_1}$ in $\A(B_{[1,n]\setminus I})$ and $\mu^\epsilon_{j_s}=\mu_j$.

  (iv)\;There is a class of bijections $\{\varphi_s\}_{s\in[0,r]}$ between face complexes of polytopes
  \[\varphi_s:\quad\mu_{i_{s}}\cdots\mu_{i_1}(N)|_{\A(B)}\quad\overset{\cong}\longrightarrow\quad \mu_{i_{s}}\cdots\mu_{i_1}(N)|_{\A(-B^\top)}\]
  such that for any $0\leqslant s\leqslant k\leqslant r$ and any face $S$ of $\mu_{i_{s}}\cdots\mu_{i_1}(N)|_{\A(B)}$,

  (1)\;$dim(S)=dim(\varphi_s(S))$. And $deg_{x_j}(p)=0$ for any point $p\in S$ if and only if $deg_{x_j}(p\p)=0$ for any point $p\p\in \varphi_s(S)$.

  (2)\;$S$ is a section of $\mu_{i_{s}}\cdots\mu_{i_1}(N)|_{\A(B)}$ if and only if $\varphi_s(S)$ is a section of $\mu_{i_{s}}\cdots\mu_{i_1}(N)|_{\A(-B^\top)}$. Hence the mutation sequence for $S$ induced by $\mu_{i_k}\cdots\mu_{i_{s+1}}$ in $\A(B)$ equals to that for $\varphi_s(S)$ induced by $\mu_{i_k}\cdots\mu_{i_{s+1}}$ in $\A(-B^\top)$.

  (3)\;let $S\p$ be the face of $\mu_{i_{k}}\cdots\mu_{i_1}(N)|_{\A(B)}$ correlating to $S$ under $\mu_{i_k}\cdots\mu_{i_{s+1}}$, then $\varphi_k(S\p)$ is the face of $\mu_{i_{k}}\cdots\mu_{i_1}(N)|_{\A(-B^\top)}$ correlating to $\varphi_s(S)$ under $\mu_{i_k}\cdots\mu_{i_{s+1}}$.
\end{Theorem}
\begin{Remark}\label{comparison with Muller's result}
  In the skew-symmetrizable case, Muller showed in \cite{M} that the support of cluster scattering diagram $\mathcal{D}_{[1,n]\setminus I}$ of sub-cluster algebra $\A(B_{[1,n]\setminus I})$ can be embedded into the support of cluster scattering diagram $\mathcal{D}$ of $\A(B)$ by times $\R^{[1,n]\setminus I}$. Since a mutation sequence $\mu$ in $\A(B)$ can be explained as a sequence of walls crossed by a curve in $\mathcal{D}$, it naturally induces a mutation sequence $\tilde{\mu}$ in $\A(B_{[1,n]\setminus I})$ by ignoring crossings not incident to walls in $\mathcal{D}_{[1,n]\setminus I}\times\R^{[1,n]\setminus I}$, or equivalently, only keeping crossings at walls whose normals become 0 under $\pi_{I}$. According to this result, the complete fan induced by $\mathcal{D}_{[1,n]\setminus I}$ under embedding can be regarded as a coarser fan obtained from that induced by $\mathcal{D}$ via gluing some cones together.

  In TSSS case, Theorem \ref{mutation of faces} claims $\tilde{\mu}$ can be determined by how a face of a polytope is changed under $\mu$. In Theorem \ref{c-vectors in N} (i) and (ii), by choosing the particular polytope $N$, we are able to relate $c$-vectors with primitive vectors of a set of edges of polytopes mutation equivalent to $N$. Therefore we can now restate Theorem \ref{mutation of faces} via $g$-vectors as Theorem \ref{c-vectors in N} (iii), a result similar to that in \cite{M} for $g$-fans together with Proposition \ref{compatibility degree and multiplication} and the fact that $\gamma_{[1,n]\setminus I;0}(N_{\pi_{[1,n]\setminus I}(h)}|_{\A(B_{[1,n]\setminus I})})$ equals the intersection of the plane $\bigcap\limits_{j\notin I}z_j=0$ and $N_h|_{\A(B)}$ for any $h\in\Z^n$. It is restricted to $g$-fans because our discussion is based on mutations of initial seed, so we could not explain crossings beyond $g$-fans in this way.

  However, on the other hand, recall the definition of the polytope $N$, which is the Newton polytope of $\prod\limits_{i=1}^n\rho_{e_i}\rho_{-e_i+[-b_i]_+}$, so during mutations we in fact only focus on cluster monomials. By taking all polytope functions into consideration, we are able to reach the mysterious region outside the $g$-fan.
\end{Remark}

\begin{Proof}[\textit{proof of Proposition \ref{sign-synchronicity of c-vectors for B and -B^T}, Theorem \ref{relation between G-matrix and C-matrix} (ii), (iii) and Theorem \ref{c-vectors in N}}]
  We take a double induction on rank $n$ and on the length of $\mu$ to prove these results simultaneously. When $n=1$, Theorem \ref{c-vectors in N} (ii) (iii) and (iv) are trivial, while the others are easy to check. Now we assume they hold for cluster algebras with rank less than $n\geqslant2$ and then consider the rank $n$ case. For the second induction, when $\mu=\emptyset$, everything is trivial. In particular, $\varphi_0$ is induced by identity for Theorem \ref{c-vectors in N} (iv). We assume they are true for $\mu\p=\mu_{i_{r-1}}\cdots\mu_{i_1}$ from $\Sigma_{t_0}$ to $\Sigma_{t\p}$ and then let $\mu=\mu_{i_r}\mu\p$.

  Next we will prove in order the following statements to complete this proof:\\
  (a)\;The inductive assumption of Proposition \ref{sign-synchronicity of c-vectors for B and -B^T} for $t\p$ induces Theorem \ref{relation between G-matrix and C-matrix} (ii) and thus (iii) for $t$.\\
  (b)\;Theorem \ref{c-vectors in N} (i) and (ii) for $\mu\p$, Proposition \ref{sign-synchronicity of c-vectors for B and -B^T} for $t\p$ and Theorem \ref{relation between G-matrix and C-matrix} (ii) for $t$ induces Theorem \ref{c-vectors in N} (i) and (ii) for $\mu$.\\
  (c)\;Theorem \ref{c-vectors in N} (i) and (ii) for $\mu$ and inductive assumption of Theorem \ref{relation between G-matrix and C-matrix} (iii) lead to Theorem \ref{c-vectors in N} (iii) for $\mu$.\\
  (d)\;Inductive assumption of Theorem \ref{relation between G-matrix and C-matrix} (iii) induces Theorem \ref{c-vectors in N} (iv) for $\varphi_r$.\\
  (e)\;Theorem \ref{c-vectors in N} (iii) and (iv) for $\mu$ and the inductive assumption of Theorem \ref{c-vectors in N} (ii), (iii) and Proposition \ref{sign-synchronicity of c-vectors for B and -B^T} induce the sign-synchronicity of $G_{t}^{B;t_0}$ and $G_{t}^{-B^\top;t_0}$ as well as that of $C_{t}^{B;t_0}$ and $C_{t}^{-B^\top;t_0}$.
  \vskip 3mm

  (a)\;Proof of this statement is exactly that of Theorem \ref{relation between G-matrix and C-matrix} (ii) and (iii).

  (b)\;From Theorem \ref{relation between G-matrix and C-matrix} (i) and the mutation formula (\ref{equation: mutation of g-vectors}) of $g$-vectors, we get the following mutation formula of $c$-vectors with respective to the mutation of initial seed: let $C_{t_0}^{\epsilon B_{t};t}=(c_{ij})_{n\times n}$ and $C_{t_0}^{\epsilon B_{t\p};t\p}=(c\p_{ij})_{n\times n}$, then
  \begin{equation}\label{equation: mutation of c-vectors in proof}
    c_{ij}=\left\{\begin{array}{ll}
                      c\p_{ij}, & \text{if $i\neq i_{r}$;} \\
                      -c\p_{i_{r}j}+\sum\limits_{s=1}^{n}[\varepsilon\epsilon b_{i_{r}s}^{t\p}]_+c\p_{sj}-\sum\limits_{s=1}^n [\varepsilon c_{si_{r};t\p}^{\epsilon B^\top;t_0}]_+\epsilon b_{sj}^{t_0}, & \text{if $i=i_{r}$.}
                    \end{array}\right.
  \end{equation}
  where $\varepsilon=\pm1$. Because of the sign-coherence of $c$-vectors, we may follow the trick Nakanishi and Zelevinsky applied in \cite{NZ}, that is, choose $\varepsilon=-sign(c_{i_{r};t\p}^{\epsilon B^\top;t_0})$ to simplify the second equation as $c_{i_{r}j}=-c\p_{i_{r}j}+\sum\limits_{s=1}^{n}[-sign(c_{i_{r};t\p}^{\epsilon B^\top;t_0})\epsilon b_{i_{r}s}^{t\p}]_+c\p_{sj}$.

  On the other hand, we can calculate how $\overrightarrow{\mu\p(L^\epsilon)}$ is changed under $\mu_{i_{r}}$. According to the mutation process of a polytope presented in Section 3, it can be checked that for any $j\in[1,n]$, denote $\overrightarrow{\mu\p(l^\epsilon_j)}=(a\p_{1j},\cdots,a\p_{nj})^\top$ and $\overrightarrow{\mu(l^\epsilon_j)}=(a_{1j},\cdots,a_{nj})^\top$, then
  \begin{equation}\label{equation: mutation of l^epsilon_j v1}
    a_{ij}=\left\{\begin{array}{ll}
                 a\p_{ij}, & \begin{split}
                                \text{if $i\neq i_{r}$;}
                             \end{split} \\\\
                 -a\p_{i_{r}j}+\sum\limits_{s=1}^n[b_{i_{r}s}^{t\p}]_+a\p_{sj}, &\begin{split}
                                                                            &\text{if $i=i_{r}$ and $\mu\p(l^\epsilon_j)$}\\
                                                                            &\text{is a bottom edge with respect to $i_{r}$;}
                                                                          \end{split}\\\\
                 -a\p_{i_{r}j}+\sum\limits_{s=1}^n[-b_{i_{r}s}^{t\p}]_+a\p_{sj}, & \begin{split}
                                                                            &\text{if $i=i_{r}$ and $\mu\p(l^\epsilon_j)$} \\
                                                                            &\text{is a top edge with respect to $i_{r}$.}
                                                                          \end{split}
               \end{array}\right.
  \end{equation}
  According to Lemma \ref{positive vector and top facet}, Theorem \ref{relation between G-matrix and C-matrix} (i) and Theorem \ref{c-vectors in N} (ii) for $\mu\p$, $\mu\p(l^\epsilon_j)$ being a bottom edge with respect to $i_{r}$ if and only if $-\epsilon sign(c_{i_{r};t\p}^{-\epsilon B_{t_0};t_0})=1$ while it being a top edge with respect to $i_{r}$ is equivalent to $\epsilon sign(c_{i_{r};t\p}^{-\epsilon B_{t_0};t_0})=1$. Therefore, (\ref{equation: mutation of l^epsilon_j v1}) can be rewritten as
  \begin{equation}\label{equation: mutation of l^epsilon_j v2}
    a_{ij}=\left\{\begin{array}{ll}
                 a\p_{ij}, & \text{if $i\neq i_{r}$;} \\
                 -a\p_{ij}+\sum\limits_{s=1}^n[-\epsilon sign(c_{i_{r};t\p}^{-\epsilon B_{t_0};t_0})b_{i_{r}s}^{t\p}]_+a\p_{sj}, & \text{if $i=i_{r}$.}
               \end{array}\right.
  \end{equation}
  Comparing (\ref{equation: mutation of c-vectors in proof}) and $(\ref{equation: mutation of l^epsilon_j v2})$, $C_{t_0}^{\epsilon B_t;t}=\overrightarrow{\mu(L^\epsilon)}$ follows from $C_{t_0}^{\epsilon B_{t\p};t\p}=\overrightarrow{\mu\p(L^\epsilon)}$ and Proposition \ref{sign-synchronicity of c-vectors for B and -B^T} for $t\p$. Then $G_{t_0}^{-\epsilon B_t^\top;t}=\epsilon\overrightarrow{\mu(L^\epsilon)}^\bot$ follows from $C_{t_0}^{\epsilon B_t;t}=\overrightarrow{\mu(L^\epsilon)}$ and Theorem \ref{relation between G-matrix and C-matrix} (iii) for $t$.

  (c)\;Following from Theorem \ref{mutation of faces}, $\mu^\epsilon_{j_s}\neq\emptyset$ if and only if there is a segment parallel to $e_{i_s}$ in the face of $\mu_{i_{s-1}}\cdots\mu_{i_1}(N)$ determined by $\{\mu_{i_{s-1}}\cdots\mu_{i_1}(l^\epsilon_i)|i\in I\}$, which is equivalent to the existence of a segment parallel to $e_{i_s}$ in the plane generated by $\{\mu_{i_{s-1}}\cdots\mu_{i_1}(l^\epsilon_i)|i\in I\}$. Therefore, $\mu^\epsilon_{j_s}\neq\emptyset$ if and only if $i_s\notin supp(\overrightarrow{\mu_{i_{s-1}}\cdots\mu_{i_1}(l^\epsilon_j)}^\bot)$ for any $j\notin I$, which is equivalent to $i_s\notin \bigcup\limits_{i\in [1,n]\setminus I}supp(g_{i;t_0}^{-\epsilon B_{t(s-1)}^\top;t(s-1)})$ by Theorem \ref{c-vectors in N} (ii) for $\mu_{i_{s-1}}\cdots\mu_{i_1}$.

  Next assume $\mu^\epsilon_{j_s}\neq\emptyset$. Let $\Gamma_{s-1}^\epsilon=(\tilde{\tau}_{s-1}^\epsilon(e_i))_{i\in I}$. Then
  $$\overrightarrow{\mu_{i_{s-1}}\cdots\mu_{i_1}(l_i^\epsilon)}= \tilde{\tau}_{s-1}^\epsilon(\overrightarrow{\mu^\epsilon_{j_{s-1}}\cdots\mu^\epsilon_{j_1}(l_i^\epsilon)|_{\A(B_{[1,n]\setminus I})}}) =\Gamma_{s-1}^\epsilon\overrightarrow{\mu^\epsilon_{j_{s-1}}\cdots\mu^\epsilon_{j_1}(l_i^\epsilon)|_{\A(B_{[1,n]\setminus I})}}$$
  for any $i\in I$. Following Theorem \ref{c-vectors in N} (i) for $\mu_{i_{s-1}}\cdots\mu_{i_1}$, we get
  $$C_{t_0}^{\epsilon B_{t(s-1)};t(s-1)}[I]=\Gamma_{s-1}^\epsilon C_{t_0}^{\epsilon B_{t\p(s-1)};t\p(s-1)}.$$
  Because $\tilde{\tau}_{s-1}^\epsilon(e_i)\in\N^n$ for $i\in I$, the polytope $\mu_{j_{s-1}}\cdots\mu_{j_1}(N_{[1,n]\setminus I})|_{\A(B_{[1,n]\setminus I})}$ lies in $\R_{\geqslant0}^{|I|}$ and there is an edge parallel to $e_{i_s}$ in the face of $\mu_{i_{s-1}}\cdots\mu_{i_1}(N)$ determined by $\{\mu_{i_{s-1}}\cdots\mu_{i_1}(l^\epsilon_i)|i\in I\}$, so there should be an index $j\in I$ such that $\tilde{\tau}_{s-1}^\epsilon(e_j)=e_{i_s}$, or equivalently, the $j$-th column of $\Gamma_{s-1}^\epsilon$ equals $e_{i_s}$. By Theorem \ref{mutation of faces}, in this case $\mu^\epsilon_{j_s}=\mu_j$.

  Since
  \begin{equation*}
    \begin{array}{ll}
      (G_{t_0}^{-\epsilon B_{t\p(s-1)}^\top;t\p(s-1)})^\top C_{t_0}^{\epsilon B_{t\p(s-1)};t\p(s-1)} & =I_{|I|\times|I|} \\
       & =(G_{t_0}^{-\epsilon B^\top_{t(s-1)};t(s-1)}[I])^\top C_{t_0}^{\epsilon B_{t(s-1)};t(s-1)}[I] \\
       & =(G_{t_0}^{-\epsilon B^\top_{t(s-1)};t(s-1)}[I])^\top\Gamma_{s-1}^\epsilon C_{t_0}^{\epsilon B_{t\p(s-1)};t\p(s-1)}
    \end{array}
  \end{equation*}
  by Theorem \ref{relation between G-matrix and C-matrix} (iii) for $t(s-1)$ and for cluster algebras whose rank is less than $n$, and $C_{t_0}^{\epsilon B_{t\p(s-1)};t\p(s-1)}$ is invertible, so
  \[(G_{t_0}^{-\epsilon B_{t\p(s-1)}^\top;t\p(s-1)})^\top=(G_{t_0}^{-\epsilon B^\top_{t(s-1)};t(s-1)}[I])^\top\Gamma_{s-1}^\epsilon.\]
  Therefore, $(G_{t_0}^{-\epsilon B_{t\p(s-1)}^\top;t\p(s-1)})^\top[j]=(G_{t_0}^{-\epsilon B^\top_{t(s-1)};t(s-1)}[I])^\top e_{i_s}=\pi_{[1,n]\setminus I}((G_{t_0}^{-\epsilon B^\top_{t(s-1)};t(s-1)})^\top[i_s])$.

  (d)\;By inductive assumption, we already get $\{\varphi_s\}_{s\in[0,r-1]}$ satisfying conditions in Theorem \ref{c-vectors in N} (iv). Define
  \[\varphi_{r}:\quad\mu(N)_{\A(B)}\quad\overset{\cong}\longrightarrow\quad \mu(N)_{\A(-B^\top)}\]
  such that for any face $S$ of $\mu(N)|_{\A(B)}$ correlating to $S\p$ of $\mu\p(N)|_{\A(B)}$ under $\mu_{i_r}$ in $\A(B)$, $\varphi_{r}(S)$ is the face of $\mu(N)|_{\A(-B^\top)}$ correlating to $\varphi_{r-1}(S\p)$ under $\mu_{i_r}$ in $\A(-B^\top)$.

  If there are more than one choices of $S\p$, then according to Theorem \ref{mutation of faces}, $S$ is not incident to edges parallel to $e_{i_r}$ and there should be three faces correlating to $S$ under $\mu_{i_r}$: a $dim(S)+1$-dimensional face $S\p_1$ of $\mu_{i_{r-1}}\cdots\mu_{i_1}(N)|_{\A(B)}$ with an edge parallel to $e_{i_r}$, the unique $dim(S)$-dimensional top face $S\p_2$ of $S\p_1$ with respect to $i_r$ and the unique $dim(S)$-dimensional bottom face $S\p_3$ of $S\p_1$ with respect to $i_r$. Then due to inductive assumption, Theorem \ref{c-vectors in N} (iv) for $\varphi_{r-1}$ holds, so $\varphi_{r-1}(S\p_1)$ is a $dim(S)+1$-dimensional face of $\mu_{i_{r-1}}\cdots\mu_{i_1}(N)|_{\A(-B^\top)}$ with an edge parallel to $e_{i_r}$, $\varphi_{r-1}(S\p_2)$ is the unique $dim(S)$-dimensional top face of $\varphi_{r-1}(S\p_1)$ with respect to $i_r$ and $\varphi_{r-1}(S\p_3)$ is the unique $dim(S)$-dimensional bottom face of $\varphi_{r-1}(S\p_1)$ with respect to $i_r$. Therefore, $\varphi_{r-1}(S\p_1)$, $\varphi_{r-1}(S\p_2)$ and $\varphi_{r-1}(S\p_3)$ correlate to the same face of $\mu(N)$ under $\mu_{i_r}$, which should be $\varphi_r(S)$.

  Conversely, if $\varphi_{r-1}(S\p)$ correlates to more than one faces of $\mu(N)|_{\A(-B^\top)}$ under $\mu_{i_r}$, then according to Theorem \ref{mutation of faces}, $\varphi_{r-1}(S\p)$ is not incident to edges parallel to $e_{i_r}$ and there should be three faces correlating to $\varphi_{r-1}(S\p)$ under $\mu_{i_r}$: a $dim(S\p)+1$-dimensional face $S^{\prime\prime}_1$ of $\mu_{i_{r}}\cdots\mu_{i_1}(N)|_{\A(-B^\top)}$ with an edge parallel to $e_{i_r}$, the unique $dim(S\p)$-dimensional top face $S^{\prime\prime}_2$ of $S^{\prime\prime}_1$ with respect to $i_r$ and the unique $dim(S\p)$-dimensional bottom face $S^{\prime\prime}_3$ of $S^{\prime\prime}_1$ with respect to $i_r$. Due to inductive assumption, Theorem \ref{c-vectors in N} (iv) for $\varphi_{r-1}$ holds, so $S\p$ is also not incident to edges parallel to $e_{i_r}$. Because by induction assumption, $deg_{x_j}(p)=0$ for any point $p\in S\p$ if and only if $deg_{x_j}(p\p)=0$ for any point $p\p\in \varphi_{r-1}(S\p)$, so there should be three faces correlating to $S\p$ under $\mu_{i_r}$: a $dim(S\p)+1$-dimensional face $S_1$ of $\mu_{i_{r}}\cdots\mu_{i_1}(N)|_{\A(B)}$ with an edge parallel to $e_{i_r}$, the unique $dim(S\p)$-dimensional top face $S_2$ of $S_1$ with respect to $i_r$ and the unique $dim(S\p)$-dimensional bottom face $S_3$ of $S_1$ with respect to $i_r$. In this case, we set $\varphi_r(S)=S^{\prime\prime}_i$ if $S=S_i$ for $i=1,2,3$.

  Therefore, $\varphi_r$ is a well-defined bijection with $dim(S)=dim(\varphi_{r}(S))$ for any face $S$ of $\mu(N)|_{\A(B)}$, which is the first part of Theorem \ref{c-vectors in N} (iv) (1).

  By definition of $\varphi_r$ and inductive assumption, $\{\varphi_s\}_{s\in[0,r]}$ satisfies Theorem \ref{c-vectors in N} (iv) (3).

  Assume there is an edge $l$ parallel to $e_{i_{s+1}}$ in $\mu_{i_s}\cdots\mu_{i_1}(N)|_{\A(B)}$ for some $s\in[1,r-1]$, then because of Theorem \ref{c-vectors in N} (iv) (2) for $\varphi_{s}$, this is equivalent to $\varphi_{s}(l)$ paralleling to $e_{i_{s+1}}$ in $\mu_{i_s}\cdots\mu_{i_1}(N)|_{\A(-B^\top)}$. So following from Theorem \ref{mutation of faces}, for any $s\leqslant k\in[0,r]$, we can see that $\mu_{i_k}\cdots\mu_{i_{s+1}}$ induces the same mutation sequence for a face $S$ of $\mu_{i_s}\cdots\mu_{i_1}(N)|_{\A(B)}$ as that for $\varphi_{s}(S)$, which shows the second part of Theorem \ref{c-vectors in N} (iv) (2) for $\{\varphi_s\}_{s\in[0,r]}$.

  For an edge $l$ parallel to $e_{s}$ in $\mu(N)|_{\A(B)}$ for some $s\in[1,n]$, if it correlates to an edge $l\p$ of $\mu\p(N)|_{\A(B)}$ parallel to $e_s$ under $\mu_{i_r}$, then similar to the above discussion about the well-definedness of $\varphi_r$, we can see that $\varphi_{r-1}(l\p)$ parallel to $e_s$ and $\varphi_r(l)$ is the edge in $\mu(N)|_{\A(-B^\top)}$ which both parallels to $e_s$ and correlates to $\varphi_{r-1}(l\p)$ under $\mu_{i_r}$.

  Otherwise choose the $\{i_r,s\}$-section $U$ of $\mu\p(N)|_{\A(B)}$ at $l\p$. Because $l\p$ does not parallel to $e_s$ but $l$ does, $U$ must be a face of $\mu\p(N)|_{\A(B)}$. Since $N$ is an $n$-cube, so due to the mutation of polytopes introduced in Section 3, there should be a face $V$ in $\mu_{i_s}\cdots\mu_{i_1}(N)|_{\A(B)}$ correlating to $U$ under $\mu_{i_{r}}\cdots\mu_{i_{s+1}}$ with $dim(V)=0$ when $s>0$ while $dim(V)\leqslant2$ when $s=0$. This can be done by applying mutations in the order of $\mu^{-1}$ to $\mu(N)|_{\A(B)}$ until get a face satisfying the above conditions. Then due to Theorem \ref{c-vectors in N} (iv) (3) and the first part of (1), there is a face $\varphi_s(V)$ in $\mu_{i_s}\cdots\mu_{i_1}(N)|_{\A(-B^\top)}$ correlating to $\varphi_{r-1}(U)$ under $\mu_{i_{r-1}}\cdots\mu_{i_{s+1}}$ with $dim(\varphi_{s}(V))=0$ when $s>0$ while $dim(\varphi_s(V))\leqslant2$ when $s=0$. Following from the second part of Theorem \ref{c-vectors in N} (iv) (2), $\mu_{i_{r-1}}\cdots\mu_{i_{s+1}}$ induces the same mutation sequence for $V$ and $\varphi_s(V)$ respectively. So according to Proposition \ref{rank 2 finite}, $\varphi_{r-1}(l)$ should also be an edge parallel to $e_s$ of the face correlating to $\varphi_{r-1}(U)$ under $\mu_{i_r}$. Therefore, any face $S$ of $\mu(N)|_{\A(B)}$ contains an edge parallel to $e_{s}$ if and only if so does $\varphi_{r}(S)$, which induces the first part of Theorem \ref{c-vectors in N} (iv) (2) for $\{\varphi_s\}_{s\in[0,r]}$.

  Finally we show the second part of Theorem \ref{c-vectors in N} (iv) (1) holds for $\{\varphi_s\}_{s\in[0,r]}$.

  Let $p$ be a vertex in $\mu(N)|_{\A(B)}$ such that $deg_{x_j}(p)=0$ for some $j\in[1,n]$. If $deg_{x_j}(p\p)=0$ for some vertex $p\p$ of $\mu\p(N)|_{\A(B)}$ correlating to $p$ under $\mu_{i_r}$, then by inductive assumption, $deg_{x_j}(\varphi_{r-1}(p\p))=0$ and $\varphi_{r-1}(p\p)$ correlates to $\varphi_r(p)$ under $\mu_{i_r}$. According to the construction of $\varphi_{r-1}$ and $\varphi_r$, $p\p$ is a top (or bottom respectively) face of $\mu\p(N)|_{\A(B)}$ if and only if $\varphi_{r-1}(p\p)$ is a top (or bottom respectively) face of $\mu\p(N)|_{\A(-B^\top)}$ and $p$ is a top (or bottom respectively) face of $\mu(N)|_{\A(B)}$ if and only if $\varphi_{r}(p)$ is a top (or bottom respectively) face of $\mu(N)|_{\A(-B^\top)}$. Therefore following the mutation of polytopes introduced in Section 3 and , $p\p$ correlates $p$ under $\mu_{i_r}$ with $deg_{x_j}(p\p)=deg_{x_j}(p)=0$ and $deg_{x_j}(\varphi_{r-1}(p\p))=0$ induces $deg_{x_j}(\varphi_r(p))=0$.

  Otherwise, either $deg_{x_j}(p\p)>0$ and $deg_{x_{i_r}}(p\p)<0$ or $deg_{x_j}(p\p)<0$ and $deg_{x_{i_r}}(p\p)>0$, we may take the $\{j,i_r\}$-section $U$ of $\mu(N)|_{\A(B)}$ at $p$. Similar to the discussion in the proof of the second part of Theorem \ref{c-vectors in N} (iv) (1), in both cases we can look at the mutations of $U$ and simultaneous mutations of $\varphi(U)$, by inductive assumption and Proposition \ref{rank 2 finite}, we should also have $deg_{x_j}(\varphi_r(p))=0$. In fact, there should be an edge $l$ incident to $p^{\prime\prime}$  parallel to $e_{j}$ or $e_{i_r}$ in $U^{\prime\prime}$, where $p^{\prime\prime}=p$, $U^{\prime\prime}=U$ or they correlate to $p$ and $U$ respectively under a mutation in direction $j$ or $i_r$. Therefore, $deg_{x_j}(p)=0$ if and only if $deg_{x_j}(\varphi_r(p))=0$ for any vertex $p$. Since $x_i$-degree linearly depends on the coordinates for any $i\in[1,n]$ and each face is a convex hull of its vertices, we get the second part of Theorem \ref{c-vectors in N} (iv) (1) for $\{\varphi_s\}_{s\in[0,r]}$.

  (e)\;Since $B$ is skew-symmetrizable when $n=2$ and Proposition \ref{sign-synchronicity of c-vectors for B and -B^T} holds for skew-symmetrizable matrices, we may assume $n\geqslant3$ here.

  For any $j\in[1,n]$, choose $I$ to be $[1,n]\setminus\{j\}$ and $\epsilon=+$. Then $\mu$ in $\A$ induces a mutation sequence $\tilde{\mu}$ from $\Sigma_{t_0}|_{\A(B_j)}$ to $\Sigma_{t'(r)}|_{\A(B_j)}$ by Theorem \ref{mutation of faces}. According to Theorem \ref{c-vectors in N} (iii) for $\mu$ and Proposition \ref{sign-synchronicity of c-vectors for B and -B^T} for vertices on the path from $t_0$ to $t$ induced by $\mu$, the mutation sequence in $\A(-B_j^\top)$ induced by $\mu$ is also $\tilde{\mu}$. Then due to Proposition \ref{sign-synchronicity of c-vectors for B and -B^T} for cluster algebras with rank $n-1$, $G_{t'(r)}^{B_j;t_0}$ is sign-synchronic to $G_{t'(r)}^{-B_j^\top;t_0}$.

  As said in Remark \ref{comparison with Muller's result}, by Theorem \ref{c-vectors in N} (iii) for $\mu$, Proposition \ref{compatibility degree and multiplication} and the fact that $N_{\pi_{j}(h)}|_{\A(B_{j})}$ equals the intersection of the hyperplane $z_j=0$ and $N_h|_{\A}$ for any $h\in\Z^n$, the cone generated by $G_{t}^{B;t_0}$ is contained in that generated by $\gamma_{j;0}(G_{t'(r)}^{B_j;t_0})$ and $\pm e_j$. This together with the row sign-coherence of $G$-matrices induces $(G_{t}^{B;t_0})_j$ is weakly sign-synchronic to $G_{t'(r)}^{B_j;t_0}$. Similarly, $(G_{t}^{-B^\top;t_0})_j$ is weak sign-synchronic to $G_{t'(r)}^{-B_j^\top;t_0}$.

  Therefore, $(G_{t}^{B;t_0})_j$ is weakly sign-synchronic to $(G_{t}^{-B^\top;t_0})_j$ since $g$-vectors are non-zero. Then due to the arbitrary choice of $j$ and the assumption $n\geqslant3$, we get that $G_{t}^{B;t_0}$ is weakly sign-synchronic to $G_{t}^{-B^\top;t_0}$. Moreover, Theorem \ref{c-vectors in N} (iii) and (iv) induces $g_{ij;t}^{B;t_0}=0$ if and only if $g_{ij;t}^{-B^\top;t_0}=0$ for any $i,j\in[1,n]$. Hence $G_{t}^{B;t_0}$ is sign-synchronic to $G_{t}^{-B^\top;t_0}$.

  Sign-synchronicity of $C_{t}^{B;t_0}$ and $C_{t}^{-B^\top;t_0}$ follows from the above sign-synchronicity and Theorem \ref{relation between G-matrix and C-matrix} (i).
\end{Proof}

By Theorem \ref{results in LP} (ii) and Theorem \ref{c-vectors in N} (ii), we can also explain a cluster as a set of polytope functions whose degree vectors corresponding to normal vectors of some polytopes of $\A(-B^\top)$. This leads to another perspective to consider cluster structures.

Let $\mathcal{N}_{g}(B)$ be the normal fan of $\bigoplus\limits_{\substack{g\in\Z^n\\\rho_g\text{ is a cluster monomial}}}N_g|_{\A(-B^\top)}$. For $\mathcal{N}_g(B)$ and $\mathcal{N}(B)$ which will be defined later, we omit the matrix $B$ if there is no risk of confusion.

(\ref{equation: mutation of l^epsilon_j v1}) shows how $\overrightarrow{\mu(l^\epsilon_j)}$ is changed under mutations. Via similar discussion, it can be generalized to a primitive vector of any edge of $\mu(N)$ for $\A(B)$. For a vertex $p$ of $\mu(N)$ and the primitive vector $a=(a_1,\cdots,a_n)^\top$ of an edge $\overline{pq}$ of $\mu(N)$ incident to $p$ and pointing to $p$, under a mutation in direction $k$, $a$ correlates to 0 if $a=\pm e_k$ and $l(\overline{pq})=-deg_{x_k}(p)$, otherwise it correlates to $a\p=(a\p_1,\cdots,a\p_n)^\top$ pointing to $p\p$ of an edge $\overline{p\p q\p}$ of $\mu_k\mu(N)$ correlating to $\overline{pq}$ under $\mu_k$, where
\begin{equation}\label{equation: vector of edges under mutation}
  a\p_j=\left\{\begin{array}{ll}
                 a_j, & \text{if }j\neq k; \\
                 -a_k, & \text{if $j=k$ and $a=\pm e_k$;} \\\\
                 -a_{k}+\sum\limits_{s=1}^n[b_{ks}^{t}]_+a_{s}, & \begin{split}
                                                                     &\text{if $j=k$, $\overline{pq}$ is a bottom}\\
                                                                     &\text{edge with respect to $k$ and $\overline{p\p q\p}$}\\
                                                                     &\text{is a top edge with respect to $k$;}
                                                                  \end{split}\\\\
                 -a_{k}+\sum\limits_{s=1}^n[-b_{ks}^{t}]_+a_{s}, & \begin{split}
                                                                      &\text{if $j=k$, $\overline{pq}$ is a top}\\
                                                                      &\text{edge with respect to $k$ and $\overline{p\p q\p}$}\\
                                                                      &\text{is a bottom edge with respect to $k$.}
                                                                   \end{split}
               \end{array}\right.
\end{equation}
for any $j\in[1,n]$ and $t$ represents the vertex connected to $t_0$ via the path induced by $\mu$.

According to (\ref{equation: vector of edges under mutation}), we can also describe how facets or dually the outer normal vectors are changed under a single mutation. For an arbitrary facet $S$ of $\mu(N)$ with primitive outer normal vector $v=(v_1,\cdots,v_n)^\top$ and $k\in[1,n]$, if $v_k=0$, then $S$ contains a segment parallel to $e_k$, so it either correlates to a facet $S\p$ of $\mu_k\mu(N)$ with primitive outer normal vector $v\p=v$ according to the first and the second equations of (\ref{equation: vector of edges under mutation}), or it correlates to a face $S\p$ of $\mu_k\mu(N)$ of codimension 2 which lies in the hyperplane $z_k=0$ and which is the intersection of a top facet $S\p_1$ of $\mu_k\mu(N)$ with respect to $k$ with primitive outer normal vector $v\p_1$ and a bottom facet $S\p_2$ of $\mu_k\mu(N)$ with respect to $k$ with primitive outer normal vector $v\p_2$, thus $v$ is the primitive vector in the intersection of the cone $cone(\{v\p_1,v\p_2\})$ and the hyperplane $z_k=0$; if $v_k>0$, then $S$ is a top facet with respective to $k$, so $S$ correlates to a bottom facet $S\p$ of $\mu_k\mu(N)$ with respect to $k$ with primitive outer normal vector $v\p=(v\p_1,\cdots,v\p_n)$, where $S\p$ can be calculated from $S$ via the first and the fourth equations of (\ref{equation: vector of edges under mutation}) and
\begin{equation*}
  v\p_j=\left\{\begin{array}{ll}
                 -v_k, & \text{if }j=k; \\
                 v_j+[-b_{kj}^t]_+v_k, & \text{otherwise.}
               \end{array}\right.
\end{equation*}
as
\begin{equation*}
  \begin{array}{ll}
    v\p\cdot a\p & =\sum\limits_{j=1}^{n}v\p_j a\p_j \\
                 & =-v_k(-a\p_{k}+\sum\limits_{s=1}^n[-b_{ks}^{t}]_+a_{s})+\sum\limits_{j\neq k}(v_j+[-b_{kj}^t]_+v_k)a_j \\
                 & =\sum\limits_{j=1}^n v_j a_j \\
                 & =0;
  \end{array}
\end{equation*}
if $v_k<0$, then dually $S$ is a bottom facet with respective to $k$, so $S$ correlates to a top facet $S\p$ of $\mu_k\mu(N)$ with respective to $k$ with primitive outer normal vector $v\p=(v\p_1,\cdots,v\p_n)$, where $S\p$ can be calculated from $S$ via the first and the third equations of (\ref{equation: vector of edges under mutation}) and
\begin{equation*}
  v\p_j=\left\{\begin{array}{ll}
                 -v_k, & \text{if }j=k; \\
                 v_j+[b_{kj}^t]_+v_k, & \text{otherwise.}
               \end{array}\right.
\end{equation*}
Therefore, when $v_k\neq0$, a facet $S$ of $\mu(N)$ with outer normal vector $v$ correlates to a facet $S\p$ of $\mu_k\mu(N)$ with outer normal vector $v\p$ under $\mu_k$ such that
\begin{equation}\label{equation: normal vector under mutation}
  v\p_j=\left\{\begin{array}{ll}
                 -v_k, & \text{if }j=k; \\
                 v_j+sign(v_k)[-b_{kj}^t v_k]_+, & \text{otherwise.}
               \end{array}\right.
\end{equation}
And vice versa since mutation is an involution.

In summary, $\mathcal{N}_g$ can be calculated as follows. For each $t\in\T_n$, $\Lambda=(\lambda_1,\cdots,\lambda_n)^{\top}\in\{\pm1\}^n$ and $n\times n$ matrix $I^{\Lambda}_n=(\lambda_1 e_1,\cdots,\lambda_n e_n)$ whose column vectors are outer normal vector of $N$ at vertex $\frac{1}{2}(\sum\limits_{i=1}^n e_i+\Lambda)$, by the above discussion, mutations can also be applied to $I^{\Lambda}_n$. So we obtain a set of matrices $\mathcal{G}_{t_0}^{\Lambda;t}=\{G_{s;t_0}^{\Lambda;t}\}$ under the mutation sequence $\mu$ induced by path
\[t=t_r\quad\frac{\;\; i_r\;\;}{\empty}\quad t_{r-1}\quad\frac{i_{r-1}}{\empty}\quad\cdots\quad\frac{\;\;i_{1}\;\;}{\empty}\quad t_0\]
from $t$ to $t_0$ in $\T_n$ such that $\mu(\Sigma_t)=\Sigma_{t_0}$ iteratively:

(i)\;$\mathcal{G}_{t}^{\Lambda;t}=\{I^{\Lambda}_n\}$.

(ii)\;Assume $\mathcal{G}_{t_{j}}^{\Lambda;t}$ has been determined. For each $G_{s;t_{j}}^{\Lambda;t}\in\mathcal{G}_{t_{j}}^{\Lambda;t}$, denote by $G_{\epsilon,s;t_{j}}^{\Lambda;t}$ the matrix obtained from $G_{s;t_{j}}^{\Lambda;t}$ by deleting columns not weakly-sign-synchronic to $\epsilon e_{i_j}$ for $\epsilon\in\{+,-\}$. For each non-zero $G_{\epsilon,s;t_{j}}^{\Lambda;t}$, we add an $(|I_{\epsilon,s;t_j}^{\Lambda;t}\sqcup I_{\pm,s;t_j}^{\Lambda;t}|)\times n$ matrix $\mu_{i_r}^{\epsilon}(G_{s;t_{j}}^{\Lambda;t})$ in $\mathcal{G}_{t_{j-1}}^{\Lambda;t}$, where $I_{\epsilon,s;t_j}^{\Lambda;t}$ represents the label set of $G_{\epsilon,s;t_{j}}^{\Lambda;t}$, $I_{\pm,s;t_{j}}^{\Lambda;t}=\{(l_1,l_2)\in I_{+,s;t_{j}}^{\Lambda;t}\times I_{-,s;t_{j}}^{\Lambda;t}|cone(G_{+,s;t_{j}}^{\Lambda;t}[l_1],G_{-,s;t_{j}}^{\Lambda;t}[l_2])\text{ is a face of }cone(G_{s;t_{j}}^{\Lambda;t})\}$ and $\mu_{i_r}^{\epsilon}(G_{s;t_{j}}^{\Lambda;t})[k]$ is obtained from $G_{\epsilon,s;t_{j}}^{\Lambda;t}[k]$ via (\ref{equation: normal vector under mutation}) with $B_t$ replaced by $-B^\top_{t_{j}}$ when $k\in I_{\epsilon,s;t_{j}}^{\Lambda;t}$, while $\mu_{i_r}^{\epsilon}(G_{s;t_{j}}^{\Lambda;t})[k]$ is the primitive vector lying in the intersection of $z_{i_{j}}=0$ and $cone(G_{+,s;t_{j}}^{\Lambda;t}[l_1],G_{-,s;t_{j}}^{\Lambda;t}[l_2])$ when $k=(l_1,l_2)\in I_{\pm,s;t_{j}}^{\Lambda;t}$.

For each $t\in\T_n$, $\{\mathcal{G}_{t_0}^{\Lambda;t}|\Lambda\in\{\pm1\}^n\}$ induces a complete fan $\mathcal{N}_g^t$ consisting of $\{cone(G_{s;t_0}^{\Lambda;t})|G_{s;t_0}^{\Lambda;t}\in\mathcal{G}_{t_0}^{\Lambda;t}\text{ and }\Lambda\in\{\pm1\}^n\}$ and their faces, which is the normal fan of the Newton polytope corresponding to the Laurent expression of $\prod\limits_{\substack{l\in[1,n]\\t\p\in\T_n}}x_{l;t\p}^{\alpha_{l;t\p}}$ in $X$ for some $\alpha_{l;t\p}\in\N$ and $\alpha_{l;t}\in\Z_{>0}$ according to the definition of $\mathcal{N}_g$ and the above discussion about how a outer normal vector of $\mu(N)$ is changed under a single mutation (the point is that we could properly choose some $\alpha_{l;t\p}$ to be nonzero when $t\p\neq t$ to avoid the case where a facet degenerates to a face with codimension 2 under a single mutation for the convenience of calculation). In this way we may rebuild $\mathcal{N}_g$.

\begin{Proposition}
  $\mathcal{N}_g$ equals the common refinement of $\{\mathcal{N}_g^t|t\in\T_n\}$.
\end{Proposition}
\begin{proof}
  Because $\mathcal{N}_g^t$ is the normal fan of the Newton polytope corresponding to the Laurent expression of $\prod\limits_{\substack{l\in[1,n]\\t\p\in\T_n}}x_{l;t\p}^{\alpha_{l;t\p}}$ in $X$ for some $\alpha_{l;t\p}\in\N$ and $\alpha_{l;t}\in\Z_{>0}$, the common refinement of $\{\mathcal{N}_g^t|t\in\T_n\}$ equals the normal fan of the Minkovski sum of all such Newton polytopes, which therefore equals the normal fan of the Newton polytope corresponding to the Laurent expression of $\prod\limits_{\substack{l\in[1,n]\\t\p\in\T_n}}x_{l;t\p}^{\alpha_{l;t\p}}$ in $X$ for some $\alpha_{l;t\p}\in\Z_{>0}$, that is, $\mathcal{N}_g$.
\end{proof}

Although the mutation formula (\ref{equation: normal vector under mutation}) of normal vectors coincides with that of $g$-vectors, those labeled by indices in $I_{\pm,s;t\p}^{\Lambda;t}$ are not necessarily $g$-vectors of $\A(B)$ or the minus of $g$-vectors of $\A(-B)$. It is the case when either $n=2$ where the vectors labeled by indices in $I_{\pm,s;t\p}^{\Lambda;t}$ belongs to $\{\pm e_1,\pm e_2\}$ or $\A(B)$ is of finite type due to Proposition \ref{N contains g-fan} and the fact $g$-fan is a complete fan in this case. This is also presented in \cite{PPPP}.

\begin{Example}
  In general, there may be a vector in $\mathcal{N}_g(B)$ which is neither a $g$-vector of $\A(B)$ nor the minus of a $g$-vector of $\A(-B)$ even when $\A(B)$ is acyclic. Let
  \[B=-B^\top=\begin{pmatrix}
                0 & 2 & -4 \\
                -2 & 0 & 2 \\
                4 & -2 & 0
              \end{pmatrix},\]
  $\Lambda=(-1,-1,1)$ and $t$ be the vertex connected to $t_0$ by path $t=t_3\quad\frac{\;\;1\;\;}{\empty}\quad t_2\quad\frac{\;\;3\;\;}{\empty}\quad t_1\quad\frac{\;\;2\;\;}{\empty}\quad t_0$. Then starting from
  \[\mathcal{G}_{t}^{\Lambda;t}=\left\{\begin{pmatrix}
                                         -1 & 0 & 0 \\
                                         0 & -1 & 0 \\
                                         0 & 0 & 1
                                       \end{pmatrix}\right\},\]
  we get
  \[\mathcal{G}_{t_2}^{\Lambda;t}=\left\{\begin{pmatrix}
                                           1 & 0 & 0 \\
                                           -2 & -1 & 0 \\
                                           0 & 0 & 1
                                         \end{pmatrix}\right\},\]
  \[\mathcal{G}_{t_1}^{\Lambda;t}=\left\{\begin{pmatrix}
                                           1 & 0 & 0 \\
                                           -2 & -1 & 2 \\
                                           0 & 0 & -1
                                         \end{pmatrix}\right\},\]
  and
  \[\mathcal{G}_{t_0}^{\Lambda;t}=\left\{\begin{pmatrix}
                                           1 & 0 & 0 \\
                                           0 & 0 & -2 \\
                                           -1 & -1 & 3
                                         \end{pmatrix},\begin{pmatrix}
                                                         -3 & -2 & 1 & 0 \\
                                                         2 & 1 & 0 & 0 \\
                                                         0 & 0 & -1 & -1
                                                       \end{pmatrix}\right\}.\]
  So $(1,0,-1)$ is a vector in $\mathcal{N}_g$, which is the outer normal vector of $x_{2;t_4}=\rho_{(0,-3,4)^\top}$, where $t_4$ is connected to $t$ by an edge labeled 2. However, $(1,0,-1)^\top$ is neither a $g$-vector of $\A(B)$ nor the minus of a $g$-vector of $\A(-B)$. In fact, $\rho_{(1,0,-1)}$ is not self-compatible, so it can not be a cluster variable.
\end{Example}

More generally, as explained in Remark \ref{comparison with Muller's result}, according to Theorem \ref{c-vectors in N}, it is natural to thereby construct a complete fan for $B$ by taking the normal fan of $\bigoplus\limits_{h\in\Z^n}N_h|_{\A(-B^\top)}$, or equivalently, by taking the common refinement of $\{\mathcal{N}(N_h|_{\A(-B^\top)})|h\in\Z^n\}$. More precisely, the normal fans of
$$\bigoplus\limits_{\substack{h\in\Z^n\\h\geqslant g\\dim(N_h|_{\A(-B^\top)})=n}}N_h|_{\A(-B^\top)}$$
for any $g\in\Z^n$ form an inverse system and denote by $\mathcal{N}(B)$ the inverse limit of this inverse system. By definition, $\mathcal{N}$ is a complete simplical pointed fan with dimension $n$.
\begin{Lemma}\label{fan N contains non-negative and non-positive cones}
  $\mathcal{N}$ contains the non-negative cone and the non-positive cone.
\end{Lemma}
\begin{proof}
  We have already know $e_i$ and $-e_i$ are normal vector of $N$ for any $i\in[1,n]$, so we only need to show a primitive vector $v$ can not be a normal vector of some polytopes when either $v\in\N^n\setminus\{e_i|i\in[1,n]\}$ or $-v\in\N^n\setminus\{e_i|i\in[1,n]\}$. This is ensured by Theorem \ref{results in LP} (iii) and the fact that there is a unique maximal point and a unique minimal point in $N_h$ for any $h\in\Z^n$ showed in \cite{LP}.
\end{proof}
\begin{Proposition}\label{N contains g-fan}
  $\mathcal{N}$ contains all cones of $g$-fan of $\A(B)$ and all cones of the central symmetry of $g$-fan of $\A(-B^\top)$.
\end{Proposition}
\begin{proof}
  We only prove the first part, the second part can be dealt with dually.

  It is clear all $g$-vectors are in $\mathcal{N}$ due to Theorem \ref{c-vectors in N} (ii), so we only need to show there is no non-$g$-vector normal vector of polytopes in the cone generated by columns of a $G$-matrix.

  Assume $v$ is a normal vector of a facet of $N_h|_{\A(-B^\top)}$ and $v$ lies in the cone generated by columns of $G_t^{B;t_0}$. Let $\mu$ be the mutation sequence induced by the path from $t_0$ to $t$ and $N\p$ be the Newton polytope of $\prod\limits_{i=1}^n\rho_{g_{i;t}^{B;t_0}}\rho_{g_{i;t_i}^{B;t_0}}$, where $t_i$ is the vertex connected to $t$ by an edge labeled $i$. Then $\mu(N\p)$ is an $n$-cube. By Theorem \ref{c-vectors in N} (ii), the cone generated by columns of $G_t^{B;t_0}$ is a cone in $\mathcal{N}(N\p)$. Let $N''=N\p\oplus N_h|_{\A(-B^\top)}$ and then consider $\mu(N'')$. Because of the row sign-coherence of $G$-matrices and Theorem \ref{c-vectors in N} (ii), under any mutations, the normal correlating to $v$ (that is, the normal of the facet correlating to the facet of which $v$ is the normal) lies in the cone generated by normals correlating to columns of $G_t^{B;t_0}$. Therefore, in $\mathcal{N}(N'')$, since the cone generated by normals correlating to columns of $G_t^{B;t_0}$ under $\mu$ is the non-negative cone, the normal correlating to $v$ under $\mu$ is a non-negative vector. Then Lemma \ref{fan N contains non-negative and non-positive cones} induces the normal correlating to $v$ equals $e_i$ for some $i\in[1,n]$, which means $v=g_{l;t}^{B;t_0}$ for some $l\in[1,n]$. Hence any cone in the $g$-fan is contained in $\mathcal{N}$.
\end{proof}

\vspace{8mm}

\textbf{Acknowledgements:}\;\emph{This paper is mainly written while the author was a research assistant in Hong Kong University, the author would like to thank Jiang-Hua Lu for her helps during his stay in Hong Kong. He also thanks Fang Li for so many suggestions, which improves this paper a lot. The author would also like to thanks Peigen Cao, Jiarui Fei, Antoine de Saint Germain, Siyang Liu, Lang Mou and Lujun Zhang for discussion.}

\emph{Research supported by the Research Grants council of the Hong Kong SAR, China GRF 17306621.}

\end{document}